\documentclass{amsart}

\usepackage[english]{babel}

\usepackage[letterpaper,top=2cm,bottom=2cm,left=3cm,right=3cm,marginparwidth=1.75cm]{geometry}

\usepackage{amsmath}
\numberwithin{equation}{section}
\usepackage{amsthm}
\usepackage{amssymb}
\usepackage{aliascnt}
\usepackage{graphicx}
\usepackage[colorlinks=true, allcolors=blue]{hyperref}

\newtheorem{theorem}{Theorem}[section]
\newaliascnt{definition}{theorem}
\newtheorem{definition}[definition]{Definition}
\aliascntresetthe{definition}
\newaliascnt{lemma}{theorem}
\newtheorem{lemma}[lemma]{Lemma}
\aliascntresetthe{lemma}
\newaliascnt{corollary}{theorem}
\newtheorem{corollary}[corollary]{Corollary}
\aliascntresetthe{corollary}
\newaliascnt{example}{theorem}

\aliascntresetthe{example}
\newaliascnt{proposition}{theorem}
\newtheorem{proposition}[proposition]{Proposition}
\aliascntresetthe{proposition}
\newaliascnt{property}{theorem}

\aliascntresetthe{property}
\newaliascnt{remark}{theorem}
\newtheorem{remark}[remark]{Remark}
\aliascntresetthe{remark}
\newaliascnt{conjecture}{theorem}

\aliascntresetthe{conjecture}

\title{Global Well-Posedness for the Benjamin-Ono Equation with Small Periodic Initial Data in Analytic Spaces}
\author{Yubo Wang}
\address{School of Mathematical Sciences, University of Chinese Academy of Sciences, Beijing 100049, China}
\email{wangyubo22@mails.ucas.ac.cn}

\begin{document}

\maketitle
%\tableofcontents
\begin{abstract}
We establish global well-posedness of the Benjamin--Ono equation for sufficiently small, zero-mean periodic initial data in the analytic Sobolev space $H^{\rho,1}_0$. The proof combines regularized commuting flows with an exponential spectral energy. A local uniform convexity property of this energy upgrades weak endpoint compactness to strong continuity, yielding a global flow in $C(\mathbb{R};H^{\rho,1}_0)$ and continuous dependence in the same analytic topology, with no dynamic loss of the prescribed analytic width.
\end{abstract}

\section{Introduction}
The Benjamin--Ono (BO) equation, originally derived to model the unidirectional propagation of internal gravity waves in deep stratified fluids, was introduced by Benjamin \cite{benjamin1967internal}, Davis, Acrivos \cite{davis1967solitary} and later refined by Ono \cite{ono1975algebraic}. It is given by
\begin{equation}\tag{BO} \label{eq:BO}
\partial_t u + \mathcal{H}\partial_{xx} u + u \partial_x u = 0.
\end{equation}
We study the Cauchy problem on the torus $\mathbb{T}=\mathbb{R}/2\pi\mathbb{Z}$:
$$u(t, x)=\sum_n\hat{u}(t,n)e^{inx},\quad u_0(x)=u(0, x)=\sum_n\hat{u}(0,n)e^{inx},$$ 
where $u$ is real-valued and $\mathcal{H}$ denotes the periodic Hilbert transform. A distinctive mathematical feature of the BO equation, distinguishing it from local dispersive models such as the Korteweg--de Vries (KdV) equation, is the non-local dispersive operator $\mathcal{H}\partial_{xx}$. Originating from the infinite-depth fluid regime, this non-locality dictates that solitary wave solutions exhibit algebraic, rather than exponential, spatial decay. We refer to the monograph by Saut \cite{saut2019benjamin}.
The BO equation is a completely integrable system. It admits a Lax pair representation, originally discovered by Nakamura \cite{nakamura1979backlund} and Bock and Kruskal \cite{bock1979two} (cf. also \cite{fokas1983inverse,coifman1990scattering}), an infinite hierarchy of conservation laws, and exact multi-soliton solutions. With the normalized $L^2$ pairing introduced below, three basic conserved quantities are the mean, momentum, and Hamiltonian,
\begin{equation}
M(u):=\langle u,1\rangle_{L^2},\qquad
P(u):=\frac12\langle u,u\rangle_{L^2},\qquad H_{BO}(u):=\frac12\langle u,\mathcal{H}u_x\rangle_{L^2}+\frac16\langle u,u^2\rangle_{L^2}.
 \end{equation}
We use the Hamiltonian convention $\partial_tu=-\partial_x\nabla H_{BO}(u)$, which gives \eqref{eq:BO}. In this paper, we investigate global well-posedness and persistence of a prescribed spatial analytic width for the periodic BO equation. Resolving the flow in the endpoint analytic topology requires overcoming quasilinear derivative losses without allowing the fixed exponential weight in the initial norm to deteriorate dynamically. We address this challenge by combining commuting flows with a geometric rigidification mechanism derived from the spectral theory of the Lax operator.
 
\subsection{Prior Work on Low-Regularity Well-Posedness}

Here we give a quick overview of the history of well-posedness for \eqref{eq:BO}, especially in low regularity Sobolev spaces. we recommend the monograph \cite{klein2021nonlinear} for a comprehensive account.

The Cauchy problem for the BO equation in Sobolev spaces $H^s$ has been extensively studied. Early proofs of well-posedness employed energy arguments. Abdelouhab, Bona, Felland, and Saut \cite{abdelouhab1989nonlocal} proved that \eqref{eq:BO} is locally well-posed for $s>\frac{3}{2}$ and globally well-posed in $H^\infty$. This period culminated in the proof that \eqref{eq:BO} is well-posed in $H^s$ for $s > \frac{3}{2}$ on $\mathbb{T}$ and for $s\ge \frac{3}{2}$ on $\mathbb{R}$ \cite{ponce1991global}. 

A fundamental technical obstruction is that the dispersion provided by $\mathcal{H}\partial_{xx}$ is exactly of second order. This dispersion is structurally insufficient to decouple the high-low frequency interactions induced by the quasilinear convective term $u\partial_x u$ via classical Strichartz estimates. Consequently, Molinet, Saut, and Tzvetkov \cite{molinet2001ill} established that the data-to-solution map is not $C^2$ smooth. Koch and Tzvetkov \cite{koch2005nonlinear} subsequently proved that the flow map fails to be locally uniformly continuous in any neighborhood of the origin for any $s \ge 0$, rendering standard Picard iteration on the Duhamel formulation inapplicable.

By incorporating Strichartz estimates into energy methods, \cite{koch2003local} advanced well-posedness on the line to $s>\frac{5}{4}$. Further refinements of this style of argument in \cite{kenig2003local} led to well-posedness for $s>\frac{9}{8}$.

Tao \cite{tao2004global} introduced a global gauge transformation to control the resonant interactions and obtained well-posedness for $s\ge1$. We refer to \cite{tao2006nonlinear} for further motivation for this gauge transform. Subsequent developments include \cite{burq2006benjamin,molinet2009well,molinet2007global,ionescu2007global}.

Following subsequent developments, the well-posedness problem was resolved in $H^s$ by Killip, Laurens, and Vi{\c{s}}an \cite{killip2024sharp}. To achieve sharp global well-posedness at $s > -1/2$, they utilized the resolvent gauge $m(\kappa, u) = (\mathcal{L}_u + \kappa)^{-1} u_+$ associated with the Lax operator $\mathcal{L}_u$. By employing $m$ as a phase correction within a paradifferential framework, they applied the method of commuting flows introduced in \cite{killip2019kdv} and developed in several subsequent papers (e.g. \cite{talbut2021benjamin}\cite{killip2023well}\cite{bringmann2021global}). Contemporaneously, G\'erard, Kappeler, and Topalov \cite{gerard2023sharp}\cite{gerard2021integrability} independently resolved the sharp Sobolev well-posedness via the construction of a global Birkhoff normal form, G\'erard \cite{gerard2023explicit} also provided an explicit formula for \eqref{eq:BO}.

\subsection{Limitations of Sobolev Theories in the Analytic Topology}

While recent breakthroughs resolve well-posedness in Sobolev spaces, analytic topologies introduce an additional issue: estimates at a fixed exponential Fourier weight generally lose width under nonlinear evolution. We work in the zero-mean analytic Sobolev space $H_0^{\rho,s}$, for $s\ge1$ and a prescribed width $\rho>0$, with norm:
\begin{equation}
\|u\|_{\rho,s}^2=\sum_{n\ne0}\langle n\rangle^{2s}|\widehat u(n)|^2e^{2\rho|n|}<\infty.
\end{equation}

Kato and Masuda \cite{kato1986nonlinear} developed a general persistence theory for analytic solutions, including BO. Gevrey energy methods, originating with Foias and Temam \cite{foias1989gevrey}, provide quantitative lower bounds for analytic widths in a variety of equations; representative dispersive examples include KdV \cite{huang2019new,selberg2015lower} and NLS \cite{tesfahun2017radius}. The question addressed here is more specific: whether a fixed width already present in the initial norm can be retained globally for small periodic BO data, without introducing a time-dependent decreasing weight. We do not claim conservation of the maximal strip of holomorphic continuation.

In the broader context of nonlinear PDEs, the scenario where $\rho(t) \ge \rho(0)$ (or maintains a uniform positive lower bound) is typically reserved for systems with a smoothing mechanism or a specific integrable structure. For dissipative systems like the Navier-Stokes equations, the balance between nonlinear frequency cascading and viscous damping allows one to prove that $\rho(t)$ remains bounded below by a positive constant $c > 0$ for all $t > T$ \cite{foias1989gevrey}, effectively "recovering" analyticity from the flow. In the purely dispersive setting, however, such persistence is far more elusive. While the KdV and NLS equations exhibit algebraic decay of the analytic radius \cite{huang2019new, tesfahun2017radius}, it is conjectured that the underlying integrability and the existence of infinite conservation laws might prevent the radius from vanishing. Notably, for certain traveling wave solutions or solitons, the profile remains unchanged, and thus $\rho(t) \equiv \rho(0)$ is trivially satisfied \cite{li1996analyticity}.

Inspired by the gauge transformation framework introduced by Killip, Laurens and Vişan, \cite{killip2024sharp}, we use the gauge transformation to "algebraize" the derivative operations within the nonlinear terms:
\begin{equation*}
\partial_x m = -\frac{i}{2} \Big( \kappa m + C_+ \big( u(m - 1) \big) \Big).
\end{equation*}
We demonstrate that the prescribed analyticity width $\rho$ is preserved globally under the flow. This phenomenon of a non-decaying analytic radius is particularly striking, as it provides a counter-narrative to the typical contraction observed in most nonlinear dispersive PDEs. However, it should be emphasized that this result is deeply rooted in the specific gauge transformation of \eqref{eq:BO}. Consequently, such a "persistence of analyticity" remains a rare property and may not be directly generalizable to other dispersive models.

\subsection{Methodology}

In this paper, we establish the global well-posedness of the BO equation in analytic Sobolev spaces. Our approach relies on two operator-theoretic mechanisms: mapping the resolvent gauge developed by Killip et al. \cite{killip2024sharp} into a Banach algebra to control the spatial derivative, and constructing a geometric rigidification mechanism via the exponential functional calculus of the Lax operator.

\textbf{1. Resolvent Gauge as an Algebraic Regularizer.} 
Building upon the resolvent gauge introduced in \cite{killip2024sharp}, we analyze its properties within the analytic Sobolev space $H_0^{\rho, s}$ ($s \ge 1$). Because this space constitutes a Banach algebra, the unbounded spatial derivative of the gauge variable can be translated entirely into bounded multilinear operations:
\begin{equation*}
\partial_x m = -\frac{i}{2} \Big( \kappa m + C_+ \big( u(m - 1) \big) \Big).
\end{equation*}
This identity expresses unbounded differential operators in terms of bounded multilinear mappings, permitting a continuous local Picard flow for the regularized $H_\kappa$ dynamics within $H^{\rho,s}_0$ and bypassing frequency truncations that are incompatible with exponential topologies.

\textbf{2. Exponential Spectral Calculus and Geometric Rigidification.}
To extend the regularized trajectories within the small-data trapping region globally without radius collapse, we define an exponential spectral energy $\mathcal{E}_\rho(u) = || \langle \mathcal{L}_u/2 \rangle e^{\frac{\rho}{2}\mathcal{L}_u} u_+ ||_{L^2}^2$. The rigorous definition and strict conservation of this functional are established using the Borel functional calculus for the self-adjoint Lax operator $\mathcal{L}_u$ and the Stieltjes inversion of its invariant spectral measure.

To bridge the topological mismatch between the spectral exponential space and the physical phase space, we evaluate the Lax intertwining operator $W(\tau)$ across the complexified analytic strip $\tau \in [0, \rho/2]$. By solving the operator ODE $W'(\tau) = -Q(\tau)W(\tau)$ as an operator-valued path in $H^1_+$, and employing Volterra integral equations (Duhamel's formula), we completely avoid unbounded operator commutators. Evaluating this operator identity over the discrete frequency lattice yields exact topological trapping inequalities governed by convergent Riemann zeta-derived series. For initial data below a specific energetic threshold, this geometric mechanism provides a static algebraic trapping region, preventing the analytic radius from shrinking.

\subsection{Main Results}
The main result establishes global well-posedness and continuous data-to-solution dependence in the canonical analytic Sobolev space $H^{\rho, 1}_0$. The intertwining operator generates two absolute geometric constants $c_1, c_2 > 0$ derived from the Riemann zeta function (evaluated in Section 4). These constants determine a geometric bounding function $f(x) = \frac{1}{\sqrt{2}} e^{-c_1 x}(x - \frac{\sqrt{2}}{2}c_2 x^2)$, which attains a strict local maximum $A_{\max} > 0$ at $x_{\max}$. A second, smaller threshold is obtained from the local uniform convexity of the exponential spectral energy.

Let $f$ and $x_{\max}$ be as in \autoref{cor:regularized_trapping}.
Choose the strict-convexity radius $r_*$ in
\autoref{thm:spectral_strict_convexity} so that $2r_*<x_{\max}$, and set
\begin{equation} \label{eq:main_small_data}
    A_*:=f(r_*),\qquad\Omega_*:=\left\{v\in H^{\rho,1}_{0}:v \text{ real-valued, } \mathcal{E}_\rho(v)^{1/2}<A_*,\quad\|v_+\|_{\rho,1}<x_{\max}\right\}.
\end{equation}
For $v\in\Omega_*$, denote by $X(v)\in[0,r_*)$ the smaller root of
\[
f(X(v))=\mathcal{E}_\rho(v)^{1/2}.
\]

\begin{theorem}
\label{thm:main_theorem}
The Benjamin--Ono equation is globally well-posed on $\Omega_*$. More precisely, for every $u_0\in\Omega_*$ there exists a unique global solution such that, for each fixed $\epsilon\in(0,\rho/3)$,
\[
u\in C(\mathbb{R};H^{\rho,1}_0)\cap C^1(\mathbb{R};H^{\rho-3\epsilon,1}_0),\qquad u(0)=u_0,
\]
and
\[
\mathcal{E}_\rho(u(t))
=\mathcal{E}_\rho(u_0),
\qquad
\sup_{t\in\mathbb{R}}\|u_+(t)\|_{\rho,1}\le X(u_0),
\qquad
\sup_{t\in\mathbb{R}}\|u(t)\|_{\rho,1}
\le\sqrt{2}\,X(u_0).
\]
In particular, the analyticity width $\rho$ is preserved globally in time. Moreover, for every $T>0$, the solution map is continuous from $\Omega_*$ into $C([-T,T];H^{\rho,1}_0)$.
\end{theorem}

\begin{remark}
While \autoref{thm:main_theorem} is formulated for the integer regularity $s=1$, the geometric rigidification framework extends to fractional indices $s > 1/2$. Fractional microlocal operations inherently introduce generic, operator-dependent constants that obscure explicit geometric estimates. However, because the exponential weight asymptotically dominates polynomial growth, for $u_0 \in H_0^{\rho, s}$ sufficiently small, it embeds continuously into $H_0^{\rho-\delta, 1}$ by conceding an arbitrarily small $\delta > 0$. Applying \autoref{thm:main_theorem} in this embedded space preserves the prescribed width $\rho-\delta$ globally, effectively bypassing fractional Leibniz complexities while preserving the exact algebraic precision of the integrability structure.
\end{remark}

\vspace{0.3cm}
\noindent \textbf{Organization of the paper.} In Section 2, we introduce the zero-mean periodic analytic Sobolev spaces and establish the necessary Banach algebra properties. Section 3 formalizes the unbounded Lax operator and utilizes its bounded resolvent to define the regularizing gauge transformation. In Section 4, we establish the exponential spectral calculus, the intertwining identity, and a local uniform convexity property of the spectral energy. Section 5 constructs the BO flow as a limit of commuting regularizations and obtains global well-posedness in the full $H^{\rho,1}_0$ topology.

\section{Notations and Preliminaries}

Throughout this paper, we consider spatial-periodic functions defined on the one-dimensional torus $\mathbb{T} = \mathbb{R}/2\pi\mathbb{Z}$. For a general periodic function $u(x)$, its Fourier series expansion is given by
$$u(x) = \sum_{n \in \mathbb{Z}} \hat{u}(n) e^{inx},$$
where $\hat{u}(n)$ are the Fourier coefficients. We use the normalized $L^2$ inner product
\begin{equation}
\langle f,g\rangle_{L^2(\mathbb{T})} := \frac{1}{2\pi}\int_{\mathbb{T}} f(x)\overline{g(x)}\,dx,
\end{equation}
so that Plancherel's identity holds without extraneous $2\pi$ factors. 

To capture the analytic regularity of solutions while preserving an $\ell^2$-based Hilbert space structure compatible with the exact spectral theory of the Lax operator, we conduct our analysis within the analytic Sobolev spaces. For a fixed analyticity radius $\rho > 0$ and a real smoothness index $s \ge 0$, we define the analytic Sobolev space $H^{\rho, s}$ as the completion of finite trigonometric polynomials under the exponentially weighted $\ell^2$-norm:
\begin{equation} \label{eq:Hs_norm}
||u||_{\rho, s}^2 = \sum_{n \in \mathbb{Z}} \langle n \rangle^{2s} |\hat{u}(n)|^2 e^{2\rho |n|},
\end{equation}
where $\langle n \rangle = (1+|n|^2)^{1/2}$. By the continuous embedding of Sobolev spaces, for $s > 1/2$, functions in $H^{\rho, s}$ admit bounded analytic continuations to the complex strip $\{ z \in \mathbb{C} : |\text{Im} z| < \rho \}$.

\begin{lemma}
\label{lemma2.1}
For any $s > 1/2$ and $\rho > 0$, the analytic Sobolev space $H^{\rho, s}$ is a Banach algebra. Specifically, there exists a constant $C_s > 0$, independent of $\rho$, such that for any $u, v \in H^{\rho, s}$, we have $uv \in H^{\rho, s}$ and:
\begin{equation}
||uv||_{\rho, s} \le C_s ||u||_{\rho, s} ||v||_{\rho, s}.
\end{equation}
\end{lemma}

\begin{proof}
By Peetre's inequality, we have $\langle n \rangle^s \le 2^{s} (\langle n-m \rangle^s + \langle m \rangle^s)$. Thus, the Fourier coefficients of the product $uv$ satisfy:
\begin{align*}
\langle n \rangle^s |\widehat{uv}(n)| e^{\rho|n|} \le 2^{s} \sum_{m \in \mathbb{Z}} \Big( &\langle n-m \rangle^s |\hat{u}(n-m)| e^{\rho|n-m|} \cdot |\hat{v}(m)| e^{\rho|m|} \\
&+ |\hat{u}(n-m)| e^{\rho|n-m|} \cdot \langle m \rangle^s |\hat{v}(m)| e^{\rho|m|} \Big).
\end{align*}
Because $s > 1/2$, the sequence $\langle k \rangle^{-s} \in \ell^2(\mathbb{Z})$. Applying the Cauchy--Schwarz inequality over the discrete lattice to the first summation yields:
\begin{align*}
&\sum_{m \in \mathbb{Z}}\langle n-m \rangle^s |\hat{u}(n-m)| e^{\rho|n-m|}|\hat{v}(m)| e^{\rho|m|} \\
&\le \left(\sum_{m \in \mathbb{Z}}\langle n-m \rangle^{2s}\langle m\rangle^{2s} |\hat{u}(n-m)|^2 e^{2\rho|n-m|}|\hat{v}(m)|^2 e^{2\rho|m|}\right)^{\frac{1}{2}}\left(\sum_{m \in \mathbb{Z}}\langle m\rangle^{-2s} \right)^{\frac{1}{2}}.
\end{align*}
Squaring this expression and summing over $n \in \mathbb{Z}$, we apply Tonelli's theorem to interchange the order of summation. This separates the respective $\ell^2$-norms:
$$ \sum_{n \in \mathbb{Z}} \sum_{m \in \mathbb{Z}} \langle n-m \rangle^{2s} |\hat{u}(n-m)|^2 e^{2\rho|n-m|} \cdot \langle m\rangle^{2s} |\hat{v}(m)|^2 e^{2\rho|m|} = ||u||_{\rho, s}^2 ||v||_{\rho, s}^2. $$
Applying a symmetric procedure to the second term, we conclude $||uv||_{\rho, s} \le C_s ||u||_{\rho, s} ||v||_{\rho, s}$, one may take $C_s = 2^{s+1} (\sum_{k \in \mathbb{Z}} \langle k \rangle^{-2s})^{1/2}$.
\end{proof}

To utilize the completely integrable structure and explicitly address the zero-mode, we restrict our phase space to the zero-mean closed subspace:
\begin{equation} \label{eq:Hs_norm_zero}
H_0^{\rho, s} = \big\{ u \in H^{\rho, s} : \hat{u}(0) = 0 \big\}.
\end{equation}
While the product of two zero-mean functions generally possesses a non-zero spatial mean, the nonlinear vector field of the Benjamin--Ono equation appears as a pure spatial derivative $\partial_x(u^2/2)$. Since spatial differentiation annihilates the zero-mode, the entire vector field maps back into the zero-mean subspace. Consequently, the condition $\hat{u}(0, t) \equiv 0$ is dynamically invariant under the flow.

We define the frequency projections $C_+$ and $C_-$ as:
$$C_+ u(x) = \sum_{n \ge 0} \hat{u}(n) e^{inx}, \quad C_- u(x) = \sum_{n < 0} \hat{u}(n) e^{inx}.$$
We denote $u_+ = C_+ u$ and $u_- = C_- u$. The positive-frequency analytic Sobolev space is defined correspondingly as $H_+^{\rho,s}:=C_+H^{\rho,s}=\left\{f\in H^{\rho,s}:\hat f(n)=0\ \text{for }n<0\right\}$. The periodic Hilbert transform $\mathcal{H}$, defined via the frequency multiplier $-i \operatorname{sgn}(n)$, satisfies the algebraic relation $\mathcal{H}u = -i(u_+ - u_-)$ on $H_0^{\rho, s}$. Furthermore, $C_{\pm}$ are orthogonal projections and are bounded with operator norm $1$ on $H^{\rho, s}$.

Following the integrability structure of the BO equation, we introduce the free Lax operator $\mathcal{L}_0 = -2i\partial_x$ acting on the positive-frequency space $H_+^{\rho, s}$. For $f \in H^{\rho, s}_+$, its action in the Fourier domain corresponds to real multiplication:
\begin{equation}
\widehat{\mathcal{L}_0 f}(n) = 2n \hat{f}(n), \quad n \ge 0.
\end{equation}
The operator $\mathcal L_0$ is nonnegative and
self-adjoint on $H_+^{\rho,s}$, with $\ker \mathcal L_0=\operatorname{span}\{1\}$. Furthermore, because $\partial_x$ is a skew-adjoint operator under the complex $L^2$ inner product, the multiplier $-2i$ renders $\mathcal{L}_0$ self-adjoint. For a positive spectral parameter $\kappa > 0$, the free resolvent is defined as $R_0(\kappa) = (\mathcal{L}_0 + \kappa)^{-1}$.

\begin{lemma}
\label{lemma2.2}
  Let $\rho > 0$, $s > 1/2$, and $\kappa > 0$. For any $u \in H^{\rho, s}$ and $f \in H_+^{\rho, s}$, the free resolvent $R_0(\kappa)$ satisfies the estimate:
  \begin{equation}
  || R_0(\kappa) C_+ \big( u f \big) ||_{\rho, s} \le \frac{C_s}{\kappa} ||u||_{\rho, s} ||f||_{\rho, s}.
  \end{equation}
\end{lemma}

\begin{proof}
Because this multiplier commutes with the Sobolev and exponential weights, the action of $R_0(\kappa)$ satisfies:
$$||R_0(\kappa) g||_{\rho, s}^2 = \sum_{n \ge 0} \frac{\langle n \rangle^{2s}}{(2n + \kappa)^2} |\hat{g}(n)|^2 e^{2\rho n} \le \frac{1}{\kappa^2} \sum_{n \ge 0} \langle n \rangle^{2s} |\hat{g}(n)|^2 e^{2\rho n} = \frac{1}{\kappa^2} ||g||_{\rho, s}^2.$$
Taking the square root yields $||R_0(\kappa) g||_{\rho, s} \le \frac{1}{\kappa} ||g||_{\rho, s}$. Setting $g = C_+(uf)$ and applying \autoref{lemma2.1}, we obtain:
$$|| R_0(\kappa) C_+ ( u f ) ||_{\rho, s} \le \frac{1}{\kappa} ||C_+ (uf)||_{\rho, s} \le \frac{1}{\kappa} ||u f||_{\rho, s} \le \frac{C_s}{\kappa} ||u||_{\rho, s} ||f||_{\rho, s}.$$
This completes the proof.
\end{proof}

A characteristic challenge in analytic spaces is that spatial differentiation acts as an unbounded operator on any fixed $H^{\rho, s}$ topology, which structurally reflects the well-known phenomenon of derivative loss. However, this unboundedness can be quantified and managed by conceding a static loss in the analyticity radius. We formalize this mechanism via the classical Cauchy estimates for analytic Sobolev spaces.

\begin{lemma}
\label{lemma2.3}
For any static loss $\epsilon \in (0, \rho)$, any $s\in\mathbb{R}$, and $f \in H^{\rho, s}$, the spatial derivative maps boundedly into the reduced space $H_0^{\rho-\epsilon, s}$ and satisfies the quantitative estimate:
\begin{equation} 
\label{eq:cauchy_est}
||\partial_x f||_{\rho-\epsilon, s} \le \frac{1}{\epsilon e} ||f||_{\rho, s}.
\end{equation}
\end{lemma}

\begin{proof}
  In the frequency domain, the spatial derivative $\partial_x$ corresponds exactly to the multiplier $in$. We directly evaluate the $H^{\rho-\epsilon, s}$ norm squared:
  \begin{align*}
  ||\partial_x f||_{\rho-\epsilon, s}^2 &= \sum_{n \ne 0} |n|^2 \langle n \rangle^{2s} |\hat{f}(n)|^2 e^{2(\rho-\epsilon)|n|} \\
  &= \sum_{n \ne 0} \Big( \langle n \rangle^{2s} |\hat{f}(n)|^2 e^{2\rho|n|} \Big) \Big( |n|^2 e^{-2\epsilon|n|} \Big).
  \end{align*}
  Elementary calculus shows that $\sup_{x > 0} x^2 e^{-2\epsilon x} = \frac{1}{\epsilon^2 e^2}$. Factoring this supremum out of the infinite summation gives:
  $$||\partial_x f||_{\rho-\epsilon, s}^2 \le \frac{1}{\epsilon^2 e^2} \sum_{n \ne 0} \langle n \rangle^{2s} |\hat{f}(n)|^2 e^{2\rho|n|} \le \frac{1}{\epsilon^2 e^2} ||f||_{\rho, s}^2.$$
  Taking the square root yields the desired estimate.
\end{proof}

\section{The Full Lax Operator and the Resolvent Gauge}

In the study of the completely integrable Benjamin-Ono hierarchy, the dynamics are governed by the Lax operator. To construct the regularized commuting flows in Section 5 and justify the geometric rigidification in Section 4, we first establish the functional-analytic properties of the spatial Lax operator.

For a real-valued, zero-mean potential $u \in H_0^{\rho, s}$ with $s \ge 1$, the full Lax operator $\mathcal{L}_u$ acting on the Hardy space of nonnegative frequencies $L^2_+(\mathbb{T})$ is defined as a perturbation of the free operator $\mathcal{L}_0$ by a Toeplitz multiplication operator:
\begin{equation} 
\label{eq:Lax_full}
\mathcal{L}_u f = \mathcal{L}_0 f + \mathcal{T}_u f = -2i\partial_x f + C_+ (u f).
\end{equation}
The free operator $\mathcal{L}_0 = -2i\partial_x$ is nonnegative and self-adjoint on $L^2_+(\mathbb{T})$ with dense domain $\mathcal{D}(\mathcal{L}_0) = H^1_+(\mathbb{T})$. Since the analytic Sobolev space $H_0^{\rho, s}$ embeds continuously into $L^\infty(\mathbb{T})$ for $s \ge 1$, the potential $u$ acts as a bounded multiplier. Consequently, the Toeplitz perturbation $\mathcal{T}_u f = C_+(uf)$ defines a bounded, symmetric linear operator on $L^2_+(\mathbb{T})$. 

By the Kato--Rellich theorem, $\mathcal{L}_u = \mathcal{L}_0 + \mathcal{T}_u$ is a closed, self-adjoint operator on $L^2_+(\mathbb{T})$ bounded from below, with domain $\mathcal{D}(\mathcal{L}_u) = H^1_+(\mathbb{T})$. This self-adjointness permits the application of the Spectral Theorem in Section 4. 

Furthermore, we will consider the restriction of $\mathcal{L}_u$ to the analytic space $H_+^{\rho, s}$. In this context, $\mathcal{L}_0$ acts as a closed, unbounded operator on $H_+^{\rho, s}$ with dense domain $\mathcal{D}_s(\mathcal{L}_0) = H_+^{\rho, s+1}$. By \autoref{lemma2.1}, the perturbation $\mathcal{T}_u$ remains a bounded linear operator on $H^{\rho, s}_+$. Thus, $\mathcal{L}_u$ acts as a closed, unbounded operator on $H_+^{\rho, s}$ with domain $\mathcal{D}_s(\mathcal{L}_u) = H_+^{\rho, s+1}$.

The auxiliary operator $\mathcal{P}_u$, densely defined on $\mathcal{D}(\mathcal{L}_0^2) = H^2_+(\mathbb{T})$ by
\begin{equation*}
  \mathcal{P}_u f = i\partial_x^2f + (C_+ \partial_xu)(C_+ f) - C_+(\partial_xu C_+ f) - C_+(u C_+ \partial_xf),
\end{equation*}
is skew-adjoint. While the integrability literature (e.g., Wu \cite{wu2016simplicity} and G\'erard \cite{gerard2021integrability}) formally identifies the isospectral evolution via the commutator equation $\partial_t \mathcal{L}_u = [\mathcal{P}_u, \mathcal{L}_u]$, we establish the associated spectral-measure conservation through resolvent generating functions and uniqueness of the Stieltjes transform, avoiding unbounded commutator manipulations.

The construction of the regularized commuting flows relies on the resolvent of the Lax operator. We first establish the existence and analytic bounds for the resolvent $R(\kappa, u) := (\mathcal{L}_u + \kappa)^{-1}$.

\begin{proposition}
\label{proposition3.1}
  Suppose $u \in H_0^{\rho, s}$ with $s \ge 1$. For any real $\kappa > 100 C_s ||u||_{\rho, s}$, $-\kappa$ belongs to the resolvent set of $\mathcal{L}_u$ acting on $H^{\rho, s}_+$. The resolvent $R(\kappa, u) := (\mathcal{L}_u + \kappa)^{-1}$ defines a bounded linear operator on $H^{\rho, s}_+$, mapping $H^{\rho, s}_+$ bijectively onto the domain $\mathcal{D}_s(\mathcal{L}_u) = H^{\rho, s+1}_+$. It is given by the convergent Neumann series in $\mathcal{B}(H^{\rho, s}_+)$:
\begin{equation} \label{eq:neumann_series}
R(\kappa, u) = R_0(\kappa) \sum_{j=0}^\infty (-1)^j \big[ \mathcal{T}_u R_0(\kappa) \big]^j,
\end{equation}
and satisfies the operator norm bound:
\begin{equation} \label{eq:resolvent_bound}
||R(\kappa, u)||_{\mathcal{B}(H^{\rho, s}_+)} \le \frac{1}{\kappa - C_s ||u||_{\rho, s}}.
\end{equation}
\end{proposition}

\begin{proof}
To construct the bounded inverse of $\mathcal{L}_u + \kappa$, we factor it using the free resolvent $R_0(\kappa) = (\mathcal{L}_0 + \kappa)^{-1}$, which is a bounded bijection from $H^{\rho, s}_+$ onto $H^{\rho, s+1}_+$. For any $f \in \mathcal{D}_s(\mathcal{L}_u) = H^{\rho, s+1}_+$, we have:
$$(\mathcal{L}_u + \kappa) f = (\mathcal{L}_0 + \kappa) f + \mathcal{T}_u f = \Big[ I + \mathcal{T}_u R_0(\kappa) \Big] (\mathcal{L}_0 + \kappa) f.$$
To evaluate $I + \mathcal{T}_u R_0(\kappa)$, we apply \autoref{lemma2.1}. For any $g \in H^{\rho, s}_+$, we have:
$$|| \mathcal{T}_u R_0(\kappa) g ||_{\rho, s} = || C_+\big(u R_0(\kappa) g\big) ||_{\rho, s} \le C_s ||u||_{\rho, s} ||R_0(\kappa) g||_{\rho, s} \le \frac{C_s}{\kappa} ||u||_{\rho, s} ||g||_{\rho, s}.$$
The condition $\kappa > 100 C_s ||u||_{\rho, s}$ implies $|| \mathcal{T}_u R_0(\kappa) ||_{\mathcal{B}(H^{\rho, s}_+)} \le \frac{C_s ||u||_{\rho, s}}{\kappa} < 1$. By the Neumann series theorem, $I + \mathcal{T}_u R_0(\kappa)$ is invertible on $H^{\rho, s}_+$. 

The resolvent is thus given by $R(\kappa, u) = R_0(\kappa) [ I + \mathcal{T}_u R_0(\kappa) ]^{-1}$. Since $R_0(\kappa)$ acts as a bounded operator from $H^{\rho, s}_+$ to $H^{\rho, s+1}_+$, the resolvent takes values in $\mathcal{D}_s(\mathcal{L}_u)$. On the other hand, $\|R_0(\kappa)\|_{\mathcal{B}(H^{\rho,s}_+)}\le\frac{1}{\kappa}$, the composition maps into $H^{\rho, s+1}_+$ and satisfies the bound:
$$||R(\kappa, u)||_{\mathcal{B}(H^{\rho, s}_+)} \le ||R_0(\kappa)||_{\mathcal{B}(H^{\rho, s}_+)} \Big|\Big| [ I + \mathcal{T}_u R_0(\kappa) ]^{-1} \Big|\Big|_{\mathcal{B}(H^{\rho, s}_+)} \le \frac{1}{\kappa} \cdot \frac{1}{1 - \frac{C_s ||u||_{\rho, s}}{\kappa}} = \frac{1}{\kappa - C_s ||u||_{\rho, s}}.$$
\end{proof}

We introduce the resolvent gauge variable. While global a priori estimates are derived via the spectral calculus in Section 4, the local Picard iteration for the commuting flows encounters spatial derivative losses. To address this obstruction, we map the potential $u$ to a gauge variable $m$ evaluated at a large spectral parameter $\kappa$, which serves as an algebraic regularizer.

\begin{definition}
\label{definition3.1}
  Let $M > 0$ be a phase-space bound, and let $\kappa$ be a fixed spectral parameter satisfying $\kappa > 100 C_s M$. For any $u \in H_0^{\rho, s}$ with $||u||_{\rho, s} \le M$, we define the gauge variable $m(\kappa, u)$ via the resolvent acting on the positive frequency projection:
  \begin{equation} \label{eq:gauge_def}
  m(\kappa, u) = R(\kappa, u) u_+ = (\mathcal{L}_u + \kappa)^{-1} u_+.
  \end{equation}
\end{definition}

We conclude this section with an analytic bound and an algebraic identity for the gauge variable. This identity demonstrates how the gauge transformation isolates the spatial derivative, enabling its evaluation through bounded multilinear operations in $H^{\rho, s}_0$.

\begin{corollary}
\label{corollary3.1}
Under the assumptions of \autoref{definition3.1}, the gauge variable $m(\kappa, u)$ is unique, belongs to the domain $\mathcal{D}_s(\mathcal{L}_u) = H_+^{\rho, s+1}$ and satisfies the bound:
\begin{equation} \label{eq:m_bound}
||m(\kappa, u)||_{\rho, s} \le \frac{||u||_{\rho, s}}{\kappa - C_s ||u||_{\rho, s}}.
\end{equation}
Moreover, $m(\kappa, u)$ satisfies the equation $(\mathcal{L}_u + \kappa)m = u_+$ in $H^{\rho, s}_+$, which expands to:
\begin{equation} \label{eq:m_eq}
-2i\partial_x m + C_+ \big( u (m - 1) \big) + \kappa m = 0.
\end{equation}  
Consequently, the spatial derivative of the gauge variable can be evaluated through bounded multilinear operations mapping into $H^{\rho, s}_0$:
\begin{equation} \label{eq:m_derivative_algebraic}
\partial_x m = -\frac{i}{2} \Big( \kappa m + C_+ \big( u(m - 1) \big) \Big).
\end{equation}
\end{corollary}

\begin{proof}
Since $R(\kappa, u)$ maps $H^{\rho, s}_+$ bijectively onto $\mathcal{D}_s(\mathcal{L}_u) = H^{\rho, s+1}_+$, we have $m \in H^{\rho, s+1}_+$. The bound \eqref{eq:m_bound} follows from the operator norm estimate in \autoref{proposition3.1}:
$$||m(\kappa, u)||_{\rho, s} = ||R(\kappa, u) u_+||_{\rho, s} \le ||R(\kappa, u)||_{\mathcal{B}(H^{\rho, s}_+)} ||u_+||_{\rho, s} \le \frac{||u||_{\rho, s}}{\kappa - C_s ||u||_{\rho, s}}.$$

To \eqref{eq:m_eq}, we apply the operator $\mathcal{L}_u + \kappa$ to both sides of \eqref{eq:gauge_def}. Since $m \in \mathcal{D}_s(\mathcal{L}_u)$, this yields:
$$-2i\partial_x m + C_+(um) + \kappa m = u_+.$$
By factoring $u_+ = C_+ u$ as $C_+(u \cdot 1)$, we rewrite the equation as:
$$-2i\partial_x m + C_+(um) - C_+(u \cdot 1) + \kappa m = 0.$$
Rearranging the terms yields \eqref{eq:m_eq}. Uniqueness follows from the invertibility of $\mathcal{L}_u + \kappa$. Finally, isolating the operator $\partial_x m$ yields \eqref{eq:m_derivative_algebraic}.
\end{proof}

\section{Spectral Calculus and Geometric Rigidification}

In this section, we construct the high-order spectral exponential energy within the functional-analytic framework. By evaluating the intertwining operator of the Lax hierarchy in the complexified domain, we demonstrate how an algebraic commutator identity yields geometric constants. This geometric evaluation bounds the analytic Sobolev radius for small data, circumventing dynamic derivative losses.

We first utilize the spectral parameter $\kappa$ large enough to construct the generating functional of the BO hierarchy via the bounded resolvent.

\begin{proposition}
\label{prop:generating_function}
Let $u\in H_0^{\rho,s}$ be real-valued, with $s\ge 1$, and suppose that $\kappa>100C_s\|u\|_{\rho,s}$. Define the spectral generating functional
\begin{equation}
\label{eq:beta_def}
    \beta(\kappa;u):=\big\langle u_+,m(\kappa,u)\big\rangle_{L^2}=\big\langle u_+,(\mathcal L_u+\kappa)^{-1}u_+\big\rangle_{L^2}.
\end{equation}
Then $\beta(\kappa;u)$ is finite and real-valued. Moreover, the map $v\mapsto\beta(\kappa;v)$ is Fr\'echet differentiable in a neighborhood of $u$ within the real-valued subspace of $H_0^{\rho,s}$, and for every real-valued $h\in H_0^{\rho,s}$,
\begin{equation}
\label{eq:beta_gradient}
    D_u\beta(\kappa;u)[h]=\Big\langle m(\kappa,u)+\overline{m(\kappa,u)}-|m(\kappa,u)|^2,h\Big\rangle_{L^2}.
\end{equation}
\end{proposition}

\begin{proof}
By \autoref{corollary3.1}, one has $m\in\mathcal D_s(\mathcal L_u)=H_+^{\rho,s+1}\subset L_+^2(\mathbb T)$, so the pairing in \eqref{eq:beta_def} is well-defined. Since $u$ is real-valued, $\mathcal L_u$ is self-adjoint on $L_+^2(\mathbb T)$. Hence its resolvent $R(\kappa,u)$ is self-adjoint, and therefore $\beta(\kappa;u)=\langle u_+,R(\kappa,u) u_+\rangle_{L^2}\in\mathbb R$.

For $h$ sufficiently small, the resolvents $R(\kappa,u+h)$ remain uniformly bounded. The second resolvent identity gives
\[
    R(\kappa,u+h)-R(\kappa,u)=-R(\kappa,u+h)\mathcal{T}_hR(\kappa,u).
\]
Consequently,
\[
\begin{aligned}
    R(\kappa,u+h)-R(\kappa,u)+R(\kappa,u)\mathcal{T}_hR(\kappa,u)&=\bigl(R(\kappa,u)-R(\kappa,u+h)\bigr)\mathcal{T}_hR(\kappa,u)\\
    &=R(\kappa,u+h)\mathcal{T}_hR(\kappa,u)\mathcal{T}_hR(\kappa,u),
\end{aligned}
\]
which is $O(\|h\|_{\rho,s}^2)$ in $\mathcal B(L_+^2)$. Thus
\begin{equation}
\label{eq:resolvent_derivative}
    D_uR(\kappa,u)[h]=-R(\kappa,u)\mathcal{T}_hR(\kappa,u).
\end{equation}

Differentiating \eqref{eq:beta_def} yields
$$D_u\beta(\kappa;u)[h]=\langle h_+,R(\kappa,u) u_+\rangle_{L^2}+\langle u_+,D_uR(\kappa,u)[h]u_+\rangle_{L^2} +\langle u_+,R(\kappa,u) h_+\rangle_{L^2}.$$
By the self-adjointness of $R(\kappa,u)$ and the identity $m=R(\kappa,u) u_+$,
the first and third terms satisfy
\[
    \langle h_+,m\rangle_{L^2}+\langle m,h_+\rangle_{L^2}=2\operatorname{Re}\langle m,h_+\rangle_{L^2}.
\]
Since $m\in L_+^2$ and $h$ is real-valued,
\[
    2\operatorname{Re}\langle m,h_+\rangle_{L^2}=\langle m+\overline m,h\rangle_{L^2}.
\]
For the middle term, \eqref{eq:resolvent_derivative} gives
\[
\begin{aligned}
    \langle u_+,D_uR(\kappa,u)[h]u_+\rangle_{L^2}=-\langle m,\mathcal{T}_hm\rangle_{L^2} =-\langle m,C_+(hm)\rangle_{L^2}.
\end{aligned}
\]
Since $C_+$ is the orthogonal projection onto $L_+^2$ and
$C_+m=m$,
\[
    \langle m,C_+(hm)\rangle_{L^2}=\langle m,hm\rangle_{L^2}=\langle |m|^2,h\rangle_{L^2},
\]
which proves \eqref{eq:beta_gradient}.
\end{proof}

While the gradient $D_u\beta[h]$ contains a non-zero spatial mean, the regularized vector fields constructed in Section 5 apply the spatial derivative $\partial_x$, which annihilates this constant mode, ensuring compatibility with the zero-mean phase space $H_0^{\rho, s}$. 

To establish the state evolution and the conservation laws along the regularized $H_\kappa$ commuting flow, we adapt the key algebraic identities established by Killip, Laurens, and Vi{\c{s}}an \cite{killip2024sharp} to our formulation.

\begin{lemma}
\label{lemma:H_kappa_properties}
Let $\kappa >100C_s\|u\|_{\rho,s}$ be a fixed spectral parameter. Let $u^\kappa$ be a classical solution of the regularized $H_\kappa$ flow:
\begin{equation}\label{eq:regularized_flow_kappa}
\partial_t u = V_\kappa(u) := -\frac{\kappa}{2} \partial_x u + \frac{\kappa^2}{2} \partial_x \Big( m(\kappa, u) + \overline{m}(\kappa, u) - |m(\kappa, u)|^2 \Big),
\end{equation}
the following properties hold:
\begin{enumerate}
    \item There exists a time-dependent, skew-adjoint operator $\mathcal{P}_\kappa(t)$ acting on the Hardy space $L^2_+(\mathbb{T})$ such that the evolution admits the modified Lax pair representation:
    \begin{align}
    \partial_t \mathcal{L}_{u^\kappa(t)} &= [\mathcal{P}_\kappa(t), \mathcal{L}_{u^\kappa(t)}], \label{eq:operator_evolution_kappa} \\
    \partial_t u^\kappa_+(t) &= \mathcal{P}_\kappa(t) u^\kappa_+(t). \label{eq:state_evolution_kappa}
    \end{align}
    \item For every positive spectral parameter $\lambda$ for which $-\lambda$ belongs to the resolvent set of $\mathcal{L}_{u^\kappa(t)}$ along the trajectory, $\beta(\lambda; u) = \langle u_+, (\mathcal{L}_u + \lambda)^{-1} u_+ \rangle_{L^2}$ is conserved along the $H_\kappa$ flow:
    \begin{equation} \label{eq:beta_conservation_kappa}
    \frac{d}{dt}\beta(\lambda; u^\kappa(t)) = 0.
    \end{equation}
\end{enumerate}
\end{lemma}

\begin{proof}
We reduce the Lax-pair statement to \cite[Proposition 5.1]{killip2024sharp}. Set
\[
s=-t,\qquad q(s,x)=-\frac12 u^\kappa(-s,x),\qquad\nu=\frac{\kappa}{2}.
\]
Denoting the Lax operator in \cite{killip2024sharp} by $\mathcal{L}^{\mathrm{KLV}}_q=-i\partial_x-C_+(q\,\cdot)$, our normalization gives $\mathcal{L}_{u^\kappa(t)}=2\mathcal{L}^{\mathrm{KLV}}_{q(-t)}$ and hence
\[
\bigl(\mathcal{L}_{u^\kappa(t)}+\kappa\bigr)^{-1}=\frac12\bigl(\mathcal{L}^{\mathrm{KLV}}_{q(-t)}+\nu\bigr)^{-1}.
\]
Consequently,
\[
m(\kappa,u^\kappa(t))=-m^{\mathrm{KLV}}(\nu,q(-t)).
\]
Since $u^\kappa$ has zero spatial mean, so does $q$. Substituting the preceding identities into \eqref{eq:regularized_flow_kappa} shows that $q$ satisfies the periodic $H_\nu$-flow of \cite[Proposition 5.1]{killip2024sharp}:
\[
\partial_s q=-\nu^2\partial_x\left(m^{\mathrm{KLV}}+\overline{m^{\mathrm{KLV}}}+\left|m^{\mathrm{KLV}}\right|^2\right)+\nu\,\partial_x q.
\]

Let $\mathcal{P}^{\mathrm{KLV}}_\nu$ denote the skew-adjoint operator constructed in \cite[Proposition 5.1]{killip2024sharp}, and define $\mathcal{P}_\kappa(t):=-\mathcal{P}^{\mathrm{KLV}}_\nu(q(-t))$. By \cite[Proposition 5.1]{killip2024sharp},
\[
\partial_s\mathcal{L}^{\mathrm{KLV}}_q=[\mathcal{P}^{\mathrm{KLV}}_\nu, \mathcal{L}^{\mathrm{KLV}}_q],\qquad\partial_s q_+=\mathcal{P}^{\mathrm{KLV}}_\nu q_+.
\]
Using $s=-t$ and
$\mathcal{L}_{u^\kappa}=2\mathcal{L}^{\mathrm{KLV}}_q$, we obtain
\[
\partial_t\mathcal{L}_{u^\kappa}=[\mathcal{P}_\kappa,\mathcal{L}_{u^\kappa}],\qquad\partial_tu^\kappa_+=\mathcal{P}_\kappa u^\kappa_+,
\]
which proves \eqref{eq:operator_evolution_kappa} and \eqref{eq:state_evolution_kappa}.

It remains to prove \eqref{eq:beta_conservation_kappa}, we repeat the proof of \cite[Theorem 5.2]{killip2024sharp}. Since $u^\kappa(t)$ is spatially smooth, $u^\kappa_+(t)\in H^\infty_+$, and the resolvent equation implies that $R(\lambda,u^\kappa)(t)u^\kappa_+(t)\in H^\infty_+$. Hence all the expressions below are well-defined on the common smooth core $H^\infty_+$. By the resolvent identity and \eqref{eq:operator_evolution_kappa},
\[
(\partial_t R(\lambda,u^\kappa))u^\kappa_+=-R(\lambda,u^\kappa)(\partial_t\mathcal L_{u^\kappa})R(\lambda,u^\kappa) u^\kappa_+=-R(\lambda,u^\kappa)[\mathcal{P}_\kappa,\mathcal L_{u^\kappa}]R(\lambda,u^\kappa) u^\kappa_+.
\]
Observe that $\mathcal L_{u^\kappa}R(\lambda,u^\kappa) u^\kappa_+=u^\kappa_+-\lambda R(\lambda,u^\kappa) u^\kappa_+$ and $R(\lambda,u^\kappa)\mathcal L_{u^\kappa}=I-\lambda R(\lambda,u^\kappa)$, we obtain
\begin{align*}
-R(\lambda,u^\kappa)[\mathcal{P}_\kappa ,\mathcal L_{u^\kappa}]R(\lambda,u^\kappa) u^\kappa_+
&=-R(\lambda,u^\kappa)\mathcal{P}_\kappa \mathcal L_{u^\kappa}R(\lambda,u^\kappa) u^\kappa_++R(\lambda,u^\kappa)\mathcal L_{u^\kappa}\mathcal{P}_\kappa R(\lambda,u^\kappa) u^\kappa_+ \\
&=-R(\lambda,u^\kappa)\mathcal{P}_\kappa (u^\kappa_+-\lambda R(\lambda,u^\kappa) u^\kappa_+)+(I-\lambda R(\lambda,u^\kappa))\mathcal{P}_\kappa R(\lambda,u^\kappa) u^\kappa_+ \\
&=\mathcal{P}_\kappa R(\lambda,u^\kappa) u^\kappa_+-R(\lambda,u^\kappa)\mathcal{P}_\kappa u^\kappa_+.
\end{align*}
Therefore,
\begin{equation}
\label{eq:resolvent_evolution_on_state}
(\partial_tR(\lambda,u^\kappa))u^\kappa_+
=
\mathcal{P}_\kappa R(\lambda,u^\kappa) u^\kappa_+
-
R(\lambda,u^\kappa)\mathcal{P}_\kappa u^\kappa_+.
\end{equation}
Therefore, using \eqref{eq:state_evolution_kappa} and the skew-adjointness of $\mathcal{P}_\kappa$,
\begin{align*}
\frac{d}{dt}\beta(\lambda;u^\kappa(t))&=\frac{d}{dt}\left\langle u^\kappa_+,R(\lambda,u^\kappa) u^\kappa_+\right\rangle_{L^2}\\
&=\left\langle\mathcal{P}_\kappa u^\kappa_+,R(\lambda,u^\kappa) u^\kappa_+\right\rangle_{L^2}+\left\langle u^\kappa_+,[\mathcal{P}_\kappa,R(\lambda,u^\kappa)]u^\kappa_+\right\rangle_{L^2}+\left\langle u^\kappa_+,R(\lambda,u^\kappa)\mathcal{P}_\kappa u^\kappa_+\right\rangle_{L^2}\\&=0.
\end{align*}
This proves \eqref{eq:beta_conservation_kappa}.
\end{proof}

\begin{proposition}
\label{prop:spectral_measure_invariance}
Let $\kappa > 100C_s\|u\|_{\rho,s}$. For any trajectory $u^\kappa(t)\in C(I;H^{\rho,1}_0)$ governed by the regularized $H_\kappa$ flow, the spectral measure associated with the Lax operator $\mathcal{L}_{u^\kappa(t)}$ and the state $u^\kappa_+(t)$ is conserved in time.
\end{proposition}

\begin{proof}
Fix a compact time interval $I$. Since $u^\kappa\in C(I;H^{\rho,1}_0)$, the quantity $M_I:=\sup_{t\in I}\|u^\kappa(t)\|_{L^\infty}$ is finite. The Toeplitz perturbation satisfies
\begin{equation*}
\|\mathcal{T}_{u^\kappa(t)}\|_{\mathcal{B}(L^2_+)}\le M_I,
\end{equation*}
whereas $\mathcal{L}_0\ge 0$ on $L^2_+$. Hence
\begin{equation*}
\mathcal L_{u^\kappa(t)}\ge-M_I.
\end{equation*}
uniformly for $t\in I$. Choose $\lambda_0>M_I$. Then $-\lambda$ belongs to the resolvent set of every $\mathcal{L}_{u^\kappa(t)}$, $t\in I$, whenever $\lambda>\lambda_0$.

For such $\lambda$, the spectral theorem gives
\begin{equation} \label{eq:stieltjes_transform}
\beta(\lambda;u^\kappa(t))=\int_{\mathbb{R}}\frac{1}{u+\lambda}\,d\mu_{u^\kappa(t)}(u).
\end{equation}
By \autoref{lemma:H_kappa_properties}, the left-hand side is independent of $t$. Therefore, for any $s,t\in I$, the Stieltjes transforms of $\mu_{u^\kappa(s)}$ and $\mu_{u^\kappa(t)}$ agree for every $\lambda\in(\lambda_0,\infty)$. Since these are finite positive Borel measures with a common lower spectral bound, uniqueness of the Stieltjes transform implies
\begin{equation*}
\mu_{u^\kappa(t)}=\mu_{u^\kappa(s)}.
\end{equation*}
As $I$ and $s,t\in I$ were arbitrary, the spectral measure is conserved for all times on which the regularized trajectory exists.
\end{proof}

We now define the exponential functional calculus via the Spectral Theorem, linking the spectral properties to the geometric rigidification mechanism.

\begin{definition}
Let $\rho > 0$ and suppose $v \in H^{\rho, 1}_0$. We define the exponential spectral energy $\mathcal{E}_\rho(v)$ in $L^2_+$ via the Borel functional calculus:
\begin{equation}
\mathcal{E}_\rho(v) := \| e^{\frac{\rho}{2}\mathcal{L}_v} v_+ \|_{L^2}^2 + \| \frac{1}{2}\mathcal{L}_v e^{\frac{\rho}{2}\mathcal{L}_v} v_+\|_{L^2}^2.
\end{equation}
\end{definition}

\begin{theorem}
\label{thm:small_data_global}
Let $\rho>0$ and let $v\in H^{\rho,1}_0$ be real-valued. For $\tau\in[0,\rho/2]$, define $Q_v(\tau)f:=\mathcal{T}_{v_{2\tau}}f=C_+\bigl(v_{2\tau}f\bigr)$, where $\widehat{v_{2\tau}}(n):=e^{2\tau n}\widehat v(n)$, and let $W_v(\tau)$ be the solution of
\begin{equation*}
    \partial_\tau W_v(\tau)=-Q_v(\tau)W_v(\tau),\qquad W_v(0)=I.
\end{equation*}
Then 
\begin{itemize}
\item [(1)] $W_v(\tau)$ is a bounded invertible operator on $L^2_+$ for every $\tau\in[0,\rho/2]$. Moreover,
\begin{equation*}
    \bigl\|W_v(\rho/2)^{\pm1}\bigr\|_{\mathcal{B}(L^2_+)}\leq e^{c_1\|v_+\|_{\rho,1}},
\end{equation*}
where $c_1=\frac12\left(\frac{\pi^2}{6}-\frac{\pi\coth(\pi)-1}{2}\right)^{1/2}$.
\item[(2)] The energy $\mathcal{E}_\rho(v)$ is finite. More precisely, define $\psi_0(v):=e^{\frac{\rho}{2}\mathcal{L}_0}v_+$, $\eta_0(v):=\frac12\mathcal{L}_0e^{\frac{\rho}{2}\mathcal{L}_0}v_+$, and the perturbed spectral coordinates $\psi_v:=e^{\frac{\rho}{2}\mathcal{L}_v}v_+$, $\eta_v:=\frac12\mathcal{L}_ve^{\frac{\rho}{2}\mathcal{L}_v}v_+$.
\begin{equation*}
\psi_v=W_v(\rho/2)^{-1}\psi_0(v),\qquad\eta_v=W_v(\rho/2)^{-1}\left(\eta_0(v)+\frac12Q_v(\rho/2)\psi_0(v)\right).
\end{equation*}
Consequently,
\begin{equation*}
\mathcal{E}_\rho(v)=\left\|W_v(\rho/2)^{-1}\psi_0(v)\right\|_{L^2}^2+\left\|W_v(\rho/2)^{-1}\left(\eta_0(v)+\frac12Q_v(\rho/2)\psi_0(v)\right)\right\|_{L^2}^2.
\end{equation*}
\item[(3)] the spectral energy and the physical analytic norm satisfy the inequalities
\begin{equation} \label{eq:transcendental_static}
\frac{1}{\sqrt{2}}e^{-c_1\|v_+\|_{\rho,1}}\left(\|v_+\|_{\rho,1}-\frac{\sqrt{2}}{2}c_2\|v_+\|_{\rho,1}^2\right)\leq\mathcal{E}_\rho(v)^{1/2}\leq e^{c_1\|v_+\|_{\rho,1}}\left(\sqrt{2}\|v_+\|_{\rho,1}+\frac{\sqrt{2}}{2}c_2\|v_+\|_{\rho,1}^2\right),
\end{equation}
where $c_2=\bigl(\pi\coth(\pi)-1\bigr)^{1/2}$.
\end{itemize}
\end{theorem}

\begin{proof}
Since $\widehat{\mathcal{L}_0f}(n)=2n\widehat f(n)$ for $n\geq0$, the free exponential coordinates satisfy
\begin{equation}
\|v_+\|_{\rho,1}^2=\|e^{\frac{\rho}{2}\mathcal{L}_0}v_+\|_{L^2}^2+\|\frac12\mathcal{L}_0e^{\frac{\rho}{2}\mathcal{L}_0}v_+\|_{L^2}^2.
\end{equation}

We next construct the intertwining operator. For every $\tau\in[0,\rho/2]$, the function $v_{2\tau}$ belongs to $H^1$. Since $H^1(\mathbb{T})\hookrightarrow L^\infty(\mathbb{T})$ and $C_+$ is an orthogonal projection on $L^2$,
\begin{equation*}
\|Q_v(\tau)\|_{\mathcal{B}(L^2_+)}\le \|v_{2\tau}\|_{L^\infty}\le C\|v_{2\tau}\|_{H^1}\le C\|v\|_{\rho,1}.
\end{equation*}
Dominated convergence in the weighted Fourier series shows that $\tau\mapsto v_{2\tau}$ is continuous in $H^1$, and hence $\tau\mapsto Q_v(\tau)$ is norm-continuous in $\mathcal{B}(L^2_+)$ and $\mathcal{B}(H^1_+)$. The Volterra equation
\begin{equation*}
W_v(\tau)=I-\int_0^\tau Q_v(s)W_v(s)\,ds
\end{equation*}
therefore has a unique solution in $C^1([0,\rho/2];\mathcal{B}(L^2_+))$. Since $v_{2\tau}\in H^1$ and $H^1$ is a Banach algebra, $Q_v(\tau)$ also acts boundedly on $H^1_+$ and is norm-continuous there. Solving the same Volterra equation in $\mathcal{B}(H^1_+)$ shows that $W_v(\tau)$ preserves $H^1_+$. Solving
\begin{equation*}
\partial_\tau Z_v(\tau)=Z_v(\tau)Q_v(\tau),\qquad Z_v(0)=I,
\end{equation*}
in both $\mathcal{B}(L^2_+)$ and $\mathcal{B}(H^1_+)$, and differentiating $Z_v(\tau)W_v(\tau)$ and $W_v(\tau)Z_v(\tau)$, shows that $Z_v(\tau)=W_v(\tau)^{-1}$ and that the inverse also preserves $H^1_+$. Gronwall's inequality gives
\begin{equation*}
\|W_v(\tau)^{\pm1}\|_{\mathcal{B}(L^2_+)}\leq\exp\left(\int_0^\tau\|Q_v(s)\|_{\mathcal{B}(L^2_+)}\,ds\right).
\end{equation*}

At the endpoint $\tau=\rho/2$, the Toeplitz norm obeys
\begin{equation}
\|Q_v(\rho/2)\|_{\mathcal{B}(L^2_+)}\leq\left(\sum_{n\neq0}\langle n\rangle^{-2}\right)^{1/2}\|v\|_{\rho,1}=\sqrt{2}c_2\|v_+\|_{\rho,1},
\label{eq:Q_bound}
\end{equation}
Furthermore,
\begin{align}
\int_0^{\rho/2}\|Q_v(\tau)\|_{\mathcal{B}(L^2_+)}\,d\tau&\leq\sum_{m>0}|\widehat v(m)|\frac{e^{\rho m}-e^{-\rho m}}{2m} \notag\\
&\leq\frac12\left(\sum_{m=1}^{\infty}\frac{1}{m^2\langle m\rangle^2}\right)^{1/2}\|v_+\|_{\rho,1}=c_1\|v_+\|_{\rho,1}.
\label{eq:W_tau_bound}
\end{align}
Here $\sum_{m=1}^{\infty}\frac{1}{m^2\langle m\rangle^2}=\sum_{m=1}^{\infty}\left(\frac1{m^2}-\frac1{1+m^2}\right)=\frac{\pi^2}{6}-\frac{\pi\coth(\pi)-1}{2}$. This proves (1).

We now establish the intertwining identity. The Volterra equation for $W_v$ yields
\begin{equation*}
e^{-\tau\mathcal{L}_0}W_v(\tau)=e^{-\tau\mathcal{L}_0}-\int_0^\tau e^{-(\tau-s)\mathcal{L}_0}\bigl(e^{-s\mathcal{L}_0}Q_v(s)\bigr)W_v(s)\,ds.
\end{equation*}
for $n\ge0$ and $f\in L_+^2$, direct Fourier calculation gives
\begin{align*}
\widehat{e^{-s\mathcal{L}_0}Q_v(s)f}(n)&=e^{-2sn}\sum_{m\ge0}\widehat v(n-m)e^{2s(n-m)}\widehat f(m)\\
&=\sum_{m\ge0}\widehat v(n-m)e^{-2sm}\widehat f(m)=\widehat{\mathcal{T}_ve^{-s\mathcal{L}_0}f}(n).
\end{align*}
Consequently,
\begin{equation*}
e^{-\tau\mathcal{L}_0}W_v(\tau)=e^{-\tau\mathcal{L}_0}-\int_0^\tau e^{-(\tau-s)\mathcal{L}_0}\mathcal{T}_ve^{-s\mathcal{L}_0}W_v(s)\,ds.
\end{equation*}
By the bounded perturbation theorem for $C_0$-semigroups (see e.g. Chapter 3 of \cite{pazy2012semigroups}), this Volterra equation has the unique solution generated by $-(\mathcal{L}_0+\mathcal{T}_v)=-\mathcal{L}_v$. Hence
\begin{equation*}
e^{-\tau\mathcal{L}_0}W_v(\tau)=e^{-\tau\mathcal{L}_v}
\end{equation*}
in the strong operator topology on $L^2_+$.

We next prove that the positive exponential spectral energy is well-defined. Set $\psi_0(\tau):=e^{\tau\mathcal{L}_0}v_+$. For $0\leq\tau\leq\rho/2$, one has $\psi_0(\tau)\in H^1_+$ and $\|\psi_0(\tau)\|_{H^1}\leq\|v_+\|_{\rho,1}$. Since $W_v(\tau)^{-1}$ preserves $H^1_+$, define
\begin{equation*}
\psi_v(\tau):=W_v(\tau)^{-1}\psi_0(\tau)\in H^1_+.
\end{equation*}
The intertwining identity implies
\begin{equation*}
e^{-\tau\mathcal{L}_v}\psi_v(\tau)=e^{-\tau\mathcal{L}_0}\psi_0(\tau)=v_+.
\end{equation*}
The self-adjoint operator $\mathcal{L}_v$ is bounded from below, so $e^{-\tau\mathcal{L}_v}$ is bounded and injective. The Borel functional calculus therefore identifies its inverse on its range with $e^{\tau\mathcal{L}_v}$. Thus
\begin{equation*}
\psi_v(\tau)=e^{\tau\mathcal{L}_v}v_+.
\end{equation*}
Moreover, $\psi_v(\tau)\in H^1_+=\mathcal{D}(\mathcal{L}_v)$, proving that $\mathcal{E}_\rho(v)$ is finite.

The maps $\psi_0(\tau)$ and $\psi_v(\tau)$ are strongly differentiable in $L^2_+$, with derivatives $\mathcal{L}_0\psi_0(\tau)$ and $\mathcal{L}_v\psi_v(\tau)$, respectively. Since
\begin{equation*}
\psi_0(\tau)=W_v(\tau)\psi_v(\tau),
\end{equation*}
the product rule gives
\begin{equation} \label{eq:exact_vector_identity}
\mathcal{L}_0\psi_0(\tau)=-Q_v(\tau)\psi_0(\tau)+W_v(\tau)\mathcal{L}_v\psi_v(\tau).
\end{equation}
Evaluating at $\tau=\rho/2$ and writing
\begin{equation*}
\psi_0=\psi_0(v),\qquad\eta_0=\eta_0(v),\qquad\psi_v=e^{\frac{\rho}{2}\mathcal{L}_v}v_+,\qquad\eta_v=\frac12\mathcal{L}_v\psi_v,
\end{equation*}
we obtain
\begin{equation*}
\psi_0=W_v\psi_v,\qquad\eta_0=W_v\eta_v-\frac12Q_v\psi_0.
\end{equation*}
Equivalently,
\begin{equation*}
\psi_v=W_v^{-1}\psi_0,
\qquad
\eta_v=W_v^{-1}\left(\eta_0+\frac12Q_v\psi_0\right),
\end{equation*}
which proves the exact spectral-coordinate representation in the statement.

From the preceding identities, we have
\begin{align*}
\|v_+\|_{\rho,1}&=\bigl(\|\psi_0\|_{L^2}^2+\|\eta_0\|_{L^2}^2\bigr)^{1/2}\\
&\leq\|\psi_0\|_{L^2}+\|\eta_0\|_{L^2}\\
&\leq\|W_v\|_{\mathcal{B}(L^2_+)}\bigl(\|\psi_v\|_{L^2}+\|\eta_v\|_{L^2}\bigr)+\frac12\|Q_v\|_{\mathcal{B}(L^2_+)}\|\psi_0\|_{L^2}\\
&\leq\sqrt{2}e^{c_1\|v_+\|_{\rho,1}}\mathcal{E}_\rho(v)^{1/2}+\frac{\sqrt{2}}{2}c_2\|v_+\|_{\rho,1}^2.
\end{align*}
Rearranging yields the lower bound in \eqref{eq:transcendental_static}. Conversely, the exact inverse-coordinate formulas imply
\begin{align*}
\mathcal{E}_\rho(v)^{1/2}&\leq\|\psi_v\|_{L^2}+\|\eta_v\|_{L^2}\\
&\leq\|W_v^{-1}\|_{\mathcal{B}(L^2_+)}\left(\|\psi_0\|_{L^2}+\|\eta_0\|_{L^2}+\frac12\|Q_v\|_{\mathcal{B}(L^2_+)}\|\psi_0\|_{L^2}\right)\\
&\leq e^{c_1\|v_+\|_{\rho,1}}\left(\sqrt{2}\|v_+\|_{\rho,1}+\frac{\sqrt{2}}{2}c_2\|v_+\|_{\rho,1}^2\right),
\end{align*}
which proves the upper bound.
\end{proof}

\begin{corollary}
\label{cor:regularized_trapping}
If $u^\kappa\in C(I;H^{\rho,1}_0)$ is a trajectory governed by the regularized $H_\kappa$ flow, then
\begin{equation*}
\mathcal{E}_\rho(u^\kappa(t))=\mathcal{E}_\rho(u^\kappa(0))\qquad\text{for every }t\in I.
\end{equation*}
Define $f(x):=\frac{1}{\sqrt{2}}e^{-c_1x}\left(x-\frac{\sqrt{2}}{2}c_2x^2\right)$, and let $x_{max}=\frac{2}{c_1+\sqrt{2}c+\sqrt{c_1^2+2c_2^2}}$ be the unique maximizer of $f$ on the interval where $f$ is positive, with $A_{max}:=f(x_{max})$. If
\begin{equation} \label{eq:small_energy_barrier}
\mathcal{E}_\rho(u^\kappa(0))^{1/2}:=A<A_{max},\qquad\|u^\kappa_+(0)\|_{\rho,1}\leq x_{max},
\end{equation}
then the equation $f(x)=A$ has a unique smaller root $x_1\in[0,x_{max})$, and
\begin{equation*}
    \sup_{t\in I}\|u^\kappa_+(t)\|_{\rho,1}\leq x_1<x_{max}.
\end{equation*}
\end{corollary}

\begin{proof}
By \autoref{prop:spectral_measure_invariance}, the spectral measure $\mu_{u^\kappa(t)}$ associated with the pair $(\mathcal{L}_{u^\kappa(t)},u^\kappa_+(t))$ is independent of $t$. By the Spectral Theorem and the finiteness proved above,
\begin{equation*}
\mathcal{E}_\rho(u^\kappa(t))=\int_{\mathbb{R}}e^{\rho\nu}\left(1+\frac{\nu^2}{4}\right)\,d\mu_{u^\kappa(t)}(\nu).
\end{equation*}
Hence
\begin{equation*}
\mathcal{E}_\rho(u^\kappa(t))=\mathcal{E}_\rho(u^\kappa(0))
\end{equation*}
for all $t\in I$.

Set $x(t):=\|u^\kappa_+(t)\|_{\rho,1}$ and $A:=\mathcal{E}_\rho(u^\kappa(0))^{1/2}$. The lower bound in \eqref{eq:transcendental_static} and conservation give
\begin{equation*}
f(x(t))\leq A
\qquad\text{for every }t\in I.
\end{equation*}
Elementary calculus shows that $f$ is strictly increasing on $[0,x_{max}]$, attains its unique positive maximum $A_{max}$ at $x_{max}$, and is strictly decreasing on $[x_{max},\frac{\sqrt{2}}{c_2}]$. Since $0\le A<A_{max}$, the equation $f(x)=A$ has two roots $x_1<x_{max}<x_2$ on that branch. The initial condition $x(0)\leq x_{max}$ and the inequality $f(x(0))\leq A$ imply $x(0)\leq x_1$. If $x(t)$ exceeded $x_1$ at some time, continuity of $t\mapsto x(t)$ would force it to enter the open interval between the two roots, where $f(x)>A$. This contradicts $f(x(t))\leq A$. Therefore
\begin{equation*}
\sup_{t\in I}x(t)\leq x_1<x_{max},
\end{equation*}
which completes the proof.
\end{proof}

\begin{lemma}
\label{lem:intertwiner_smooth_dependence}
For every $R>0$, the maps $u\longmapsto W_u(\cdot)$ and $u\longmapsto W^{-1}_u(\cdot)$ are real analytic from the ball $B_R(H^{\rho,1}_0)\cap\mathbb{R}$ into $C([0,\rho/2];\mathcal{B}(L^2_+))$. In particular, for $j=1,2$ there exists $C_{j,R}>0$ such that
\begin{align*}
\sup_{0\le\tau\le \rho/2}\|D^jW_u(\tau)[h_1,\ldots,h_j]\|_{\mathcal{B}(L^2_+)}&\le C_{j,R}\prod_{\ell=1}^j\|h_\ell\|_{\rho,1},\\
\sup_{0\le\tau\le \rho/2}\|D^jW^{-1}_u(\tau)[h_1,\ldots,h_j]\|_{\mathcal{B}(L^2_+)}
&\le C_{j,R}\prod_{\ell=1}^j\|h_\ell\|_{\rho,1}.
\end{align*}
Moreover,
\begin{equation} \label{eq:W_near_identity}
\sup_{0\le\tau\le \rho/2}\bigl(\|W_u(\tau)-I\|+\|W^{-1}_u(\tau)-I\|\bigr)\le C_{R}\|u\|_{\rho,1}.
\end{equation}
\end{lemma}

\begin{proof}
We only prove the $W_u(\tau)$ case since $W^{-1}_u(\tau)$ holds via the same argument. The estimates in the proof of \autoref{thm:small_data_global} give
\begin{equation} \label{eq:Q_integrated_linear}
\alpha(h):=\int_0^{\rho/2}\|Q_h(\tau)\|_{\mathcal{B}(L^2_+)}\,d\tau\le c_1\|h_+\|_{\rho,1}.
\end{equation}
Iterating the Volterra equation yields the Dyson series
\begin{equation*}
W_u(\tau)=I+
\sum_{n=1}^{\infty}(-1)^n
\int_{0<s_n<\cdots<s_1<\tau}
Q_u(s_1)\cdots Q_u(s_n)\,ds_n\cdots ds_1.
\end{equation*}
The $n$-th term is bounded uniformly by $\alpha(u)^n/n!$. Hence the series, together with every Fr\'echet derivative in $u$, converges uniformly on bounded subsets of $H^{\rho,1}_0$. This proves that $u\mapsto W_u(\cdot)$ real analytic.

For later use, we record the first two variational equations. The first derivative satisfies
\begin{equation*}
\partial_\tau DW_u[h]=-Q_uDW_u[h]-Q_hW_u,\qquad DW_u[h](0)=0.
\end{equation*}
Consider $W_u^{-1}DW_u[h]$, then
$$\begin{aligned}
\partial_\tau (W_u^{-1}DW_u[h])&=\partial_\tau W_u^{-1}DW_u[h]+W_u^{-1}\partial_\tau DW_u[h]\\
&=-W_u^{-1}Q_hW_u,
\end{aligned}$$
which gives
\begin{align} \label{eq:DW_formula}
W_u^{-1}(\tau)DW_u(\tau)[h]&=-\int_0^\tau W_u(s)^{-1}Q_h(s)W_u(s)\,ds.
\end{align}
Consequently,
$$\begin{aligned}
\sup_\tau\|DW_u(\tau)[h]\|&\le\int_0^\tau \|W_u(\tau)W_u(s)^{-1}\|\|Q_h(s)\|\|W_u(s)\|\,ds\\
&\le e^{\alpha(u)}\alpha(h).
\end{aligned}$$
The second derivative satisfies
\begin{align*}
\partial_\tau D^2W_u[h,k]=-Q_uD^2W_u[h,k]-Q_hDW_u[k]-Q_kDW_u[h],\qquad D^2W_u[h,k](0)={}0,
\end{align*}
the same argument implies
\begin{equation*}
\sup_\tau\|D^2W_u(\tau)[h,k]\|\le 2e^{2\alpha(u)}\alpha(h)\alpha(k).
\end{equation*}
On $B_R(H^{\rho,1}_0)\cap\mathbb{R}$, these estimates and \eqref{eq:Q_integrated_linear} give the asserted bounds for $W_u$. Finally, the Volterra equations for $W_u$ yields
\begin{align*}
\sup_\tau\|W_u(\tau)-I\|\le\sup_\tau\int_0^\tau \|Q_u(s)\|\|W(s)\|\,ds\le e^{\alpha(u)}\alpha(u).
\end{align*}
Using \eqref{eq:Q_integrated_linear} on $B_R(H^{\rho,1}_0)\cap\mathbb{R}$ proves \eqref{eq:W_near_identity}.
\end{proof}

\begin{definition}
Let $\mathcal{Y}:=L^2_+\times L^2_+$ with its product Hilbert norm and $u\in H^{\rho,1}_0\cap\mathbb{R}$, we define the linear free-coordinate map
\begin{equation*}
\mathcal{B}_\rho u
:=\left(B_0u,B_1u\right)
:=\left(e^{\frac{\rho}{2}\mathcal{L}_0}u_+,\frac12\mathcal{L}_0e^{\frac{\rho}{2}\mathcal{L}_0}u_+\right).
\end{equation*}
Thus $\|\mathcal{B}_\rho u\|_{\mathcal{Y}}=\|u_+\|_{\rho,1}$. Define
\begin{equation} \label{eq:Phi_definition}
\Phi_\rho(u):=\left(W_u(\rho/2)^{-1}B_0u,W_u(\rho/2)^{-1}\left(B_1u+\frac12Q_u(\rho/2)B_0u\right)\right).
\end{equation}
Then $\mathcal{E}_\rho(u)=\|\Phi_\rho(u)\|_{\mathcal{Y}}^2$.

\end{definition}

\begin{proposition}
\label{prop:spectral_coordinates_near_identity}
The map $\Phi_\rho:H_0^{\rho,1}\cap\mathbb{R}\to\mathcal{Y}$ is real analytic. For every $R>0$, there exists $C_R>0$ such that, for $u\in B_R(H_0^{\rho,1})$ and $h,k\in H_0^{\rho,1}$,
\begin{align}
\|\Phi_\rho(u)-\mathcal{B}_\rho u\|_{\mathcal{Y}}
&\le C_R\|u\|_{\rho,1}^2, \label{eq:Phi_quadratic}\\
\|D\Phi_\rho(u)[h]-\mathcal{B}_\rho h\|_{\mathcal{Y}}
&\le C_R\|u\|_{\rho,1}\|h\|_{\rho,1}, \label{eq:DPhi_linear_error}\\
\|D^2\Phi_\rho(u)[h,k]\|_{\mathcal{Y}}
&\le C_R\|h\|_{\rho,1}\|k\|_{\rho,1}. \label{eq:D2Phi_bound}
\end{align}
\end{proposition}

\begin{proof}
The endpoint estimate \eqref{eq:Q_bound} reads $\|Q_u(\rho/2)\|_{\mathcal{B}(L^2_+)}\le c_2\|u\|_{\rho,1}$. Recall that $B_0$ and $B_1$ are bounded linear maps on $H_0^{\rho,1}\cap\mathbb{R}\mapsto L_2^+$, together with \autoref{lem:intertwiner_smooth_dependence}, this proves real analyticity of $\Phi_\rho$.

For the first component of \eqref{eq:Phi_quadratic},
\begin{equation*}
W_u(\rho/2)^{-1}B_0u-B_0u=(W_u(\rho/2)^{-1}-I)B_0u,
\end{equation*}
then
\begin{align}\label{eq:bound}
\|(W_u(\rho/2)^{-1}-I)B_0u\|_{L^2}\le\|B_0u\|_{L^2}\int_0^{\rho/2}\|W^{-1}_u(s)\|_{\mathcal{B}(L^2_+)}\|Q_u(s)\|_{\mathcal{B}(L^2_+)}ds\le C_R\|u\|^2_{\rho,1}.
\end{align}
and for the second component,
\begin{equation*}
W_u(\rho/2)^{-1}\left(B_1u+\frac12Q_u(\rho/2)B_0u\right)-B_1u
=(W_u(\rho/2)^{-1}-I)B_1u+\frac12W_u(\rho/2)^{-1}Q_u(\rho/2)B_0u,
\end{equation*}
the same argument gives \eqref{eq:Phi_quadratic}.

For the first derivative \eqref{eq:DPhi_linear_error}, the first component satisfies
\begin{equation*}
D(W_u(\rho/2)^{-1}B_0u)[h]=D(W_u(\rho/2)^{-1})[h]B_0u+W_u(\rho/2)^{-1}B_0h.
\end{equation*}
Then 
$$\|D(W_u(\rho/2)^{-1}B_0u)[h]-W_u(\rho/2)^{-1}B_0h\|_{L^2}\le\|D(W_u(\rho/2)^{-1})[h]B_0u\|_{L^2}\le C_R\|u\|_{\rho,1}\|h\|_{\rho,1}.$$
For the second component, set $F(u):=B_1u+\frac12Q_u(\rho/2)B_0u$. Since $u\mapsto Q_u(\rho/2)$ is linear,
\begin{equation*}
DF(u)[h]
=B_1h+\frac12Q_h(\rho/2)B_0u+\frac12Q_u(\rho/2)B_0h.
\end{equation*}
Thus
\begin{equation*}
D(W_u(\rho/2)^{-1}F(u))[h]=DW_u(\rho/2)^{-1}[h]F(u)+W_u(\rho/2)^{-1}DF(u)[h].
\end{equation*}
Subtracting the corresponding components of $\mathcal{B}_\rho h$ and using \autoref{lem:intertwiner_smooth_dependence}, \eqref{eq:bound} proves \eqref{eq:DPhi_linear_error}.

Finally, we estimate \eqref{eq:D2Phi_bound} as
\begin{equation*}
D^2(W_u(\rho/2)^{-1}B_0u)[h,k]=D^2W_u(\rho/2)^{-1}[h,k]B_0u+DW_u(\rho/2)^{-1}[h]B_0k+DW_u(\rho/2)^{-1}[k]B_0h,
\end{equation*}
and
\begin{equation*}
D^2F(u)[h,k]
=\frac12Q_h(\rho/2)B_0k+\frac12Q_k(\rho/2)B_0h.
\end{equation*}
Consequently,
\begin{align*}
D^2(W_u(\rho/2)^{-1}F(u))[h,k]
={}&D^2W_u(\rho/2)^{-1}[h,k]F(u)+DW_u(\rho/2)^{-1}[h]DF(u)[k]\\
&+DW_u(\rho/2)^{-1}[k]DF(u)[h]+W_u(\rho/2)^{-1}D^2F(u)[h,k].
\end{align*}
The bounds from \autoref{lem:intertwiner_smooth_dependence} give \eqref{eq:D2Phi_bound}.
\end{proof}

\begin{theorem}
\label{thm:spectral_strict_convexity}
There exist $r_*>0$ and $\lambda_*>0$, such that the spectral energy $\mathcal{E}_\rho$ is $C^2$ on $B_{2r_*}(H_0^{\rho,1})\cap\mathbb{R}$ and
\begin{equation} \label{eq:energy_hessian_coercive}
D^2\mathcal{E}_\rho(u)[h,h]
\ge \lambda_*\|h\|_{\rho,1}^2
\end{equation}
for every $u\in B_{2r_*}(H_0^{\rho,1})\cap\mathbb{R}$ and $h\in H_0^{\rho,1}\cap\mathbb{R}$. Hence, for $u,v\in \overline{B}_{\sqrt{2}r_*}(H_0^{\rho,1})\cap\mathbb{R}$,
\begin{equation} \label{eq:strong_convexity_inequality}
\mathcal{E}_\rho(v)
\ge \mathcal{E}_\rho(u)
+D\mathcal{E}_\rho(u)[v-u]
+\frac{\lambda_*}{2}\|v-u\|_{\rho,1}^2.
\end{equation}
In particular, $\mathcal{E}_\rho$ is weakly lower semicontinuous on $\overline{B}_{\sqrt{2}r_*}(H_0^{\rho,1})\cap\mathbb{R}$ and has the following Kadec property there: if
\begin{equation*}
u_n\rightharpoonup u\quad\text{in }H_0^{\rho,1},
\qquad
\mathcal{E}_\rho(u_n)\to\mathcal{E}_\rho(u),
\end{equation*}
with $u_n$, $u\in\overline{B}_{\sqrt{2}r_*}(H_0^{\rho,1})\cap\mathbb{R}$, then $u_n\to u$ strongly in $H_0^{\rho,1}\cap\mathbb{R}$.
\end{theorem}

\begin{proof}
We begin with the fact $\mathcal{E}_\rho(u)=\|\Phi_\rho(u)\|_{\mathcal{Y}}^2$. With the complex inner product interpreted through its real part,
\begin{equation} \label{eq:energy_hessian_formula}
D^2\mathcal{E}_\rho(u)[h,h]
=2\|D\Phi_\rho(u)[h]\|_{\mathcal{Y}}^2+2\operatorname{Re}\langle\Phi_\rho(u),D^2\Phi_\rho(u)[h,h]\rangle_{\mathcal{Y}}. 
\end{equation}
Fix $R_0>0$. By \eqref{eq:DPhi_linear_error}, for $u\in B_{R_0}$,
\begin{equation*}
\|D\Phi_\rho(u)[h]\|_{\mathcal{Y}}
\ge\|\mathcal{B}_\rho h\|_{\mathcal{Y}}-\|D\Phi_\rho(u)[h]-\mathcal{B}_\rho h\|_{\mathcal{Y}}\ge (\frac{1}{\sqrt{2}}-C_{R_0}\|u\|_{\rho,1})\|h\|_{\rho,1}.
\end{equation*}
Moreover, \eqref{eq:Phi_quadratic} and \eqref{eq:D2Phi_bound} imply
\begin{align*}
\operatorname{Re}\langle\Phi_\rho(u),D^2\Phi_\rho(u)[h,h]\rangle_{\mathcal{Y}}\ge-\|\Phi_\rho(u)\|_{\mathcal{Y}}\|D^2\Phi_\rho(u)[h,h]\|_{\mathcal{Y}}\ge -C_{R_0}(1+C_{R_0}\|u\|_{\rho,1})\|u\|_{\rho,1}\|h\|_{\rho,1}^2.
\end{align*}
Substituting these estimates into \eqref{eq:energy_hessian_formula} and choose $R_0$ small enough such that $\frac{1}{\sqrt{2}}-C_{R_0}\|u\|_{\rho,1}\ge0$, we obtain
\begin{align*}
D^2\mathcal{E}_\rho(u)[h,h]\ge\Bigg[2\left(\frac{1}{\sqrt{2}}-C_{R_0}\|u\|_{\rho,1}\right)^2-2C_{R_0}\left(1+C_{R_0}\|u\|_{\rho,1}\right)\|u\|_{\rho,1}\Bigg]\|h\|_{\rho,1}^2.
\end{align*}
The coefficient in brackets is continuous in $\|u\|_{\rho,1}$ and equals $1$ at the origin. We may therefore choose $r_*>0$ sufficiently small, with $2r_*<\min\{R_0,x_{\max}\}$,and $\lambda_*>0$ so that \eqref{eq:energy_hessian_coercive} holds on $B_{2r_*}$.

For $u,v\in\overline{B}_{\sqrt{2}r_*}$, set $w=v-u$ and $g(s)=\mathcal{E}_\rho(u+sw)$. The segment lies in $\overline{B}_{\sqrt{2}r_*}\subset B_{2r_*}$, and hence
\begin{equation*}
g''(s)\ge\lambda_*\|w\|_{\rho,1}^2.
\end{equation*}
Taylor's formula with integral remainder gives \eqref{eq:strong_convexity_inequality}.

It remains to prove the Kadec property. If $u_n\rightharpoonup u$ in $H_0^{\rho,1}$ with $u_n,u\in\overline{B}_{\sqrt{2}r_*}$, then the continuous linear functional $D\mathcal{E}_\rho(u)$ satisfies
\begin{equation*}
D\mathcal{E}_\rho(u)[u_n-u]\to0.
\end{equation*}
Taking the lower limit in \eqref{eq:strong_convexity_inequality} proves weak lower semicontinuity. If the energies also converge, the same inequality gives
\begin{equation*}
\|u_n-u\|_{\rho,1}^2\to0,
\end{equation*}
which is the Kadec property.
\end{proof}

\begin{remark}
It is instructive to consider whether the geometric rigidification mechanism established in \autoref{thm:small_data_global} can be extended to the quasi-periodic regime. For a quasi-periodic potential governed by a rationally independent frequency vector $\omega \in \mathbb{R}^\nu$ ($\nu \ge 2$), the spatial frequencies transition to the dense module $\{m \cdot \omega : m \in \mathbb{Z}^\nu\}$.

In this setting, evaluating the fractional decay generated by the Lax intertwining operator across the complexified analytic strip necessitates the convergence of the sum:
\begin{equation*}
\sum_{m \in \mathbb{Z}^\nu \setminus \{0\}} \frac{1}{(m \cdot \omega)^2 \big( 1 + (m \cdot \omega)^2 \big)}.
\end{equation*}

Due to the small divisor problem, the inner products $m \cdot \omega$ accumulate close to zero. Even if a Diophantine condition is imposed, the singularity $(m \cdot \omega)^{-2}$ introduces polynomial growth, rendering the infinite summation divergent for $\nu \ge 2$. 

This divergence reflects the spectral properties of the Lax operator. In the periodic setting, the free Lax operator restricted to nonnagative frequencies possesses a positive lower bound, providing a spectral gap. In contrast, under standard Diophantine conditions, the spectrum of the Lax operator for the quasi-periodic BO equation is devoid of spectral gaps. This gapless structure invalidates the static operator integration mechanism and precludes the application of classical inverse spectral machinery. For instance, the global inverse spectral methods utilized by Damanik and Goldstein \cite{damanik2016existence} for the quasi-periodic KdV equation rely on the presence of a gap structure in the spectrum of the Schr\"odinger operator. Because the quasi-periodic BO Lax operator lacks these gaps, such inverse spectral techniques are inapplicable. Although significant progress has been made in the local well-posedness of initial values for quasi-periodic systems in the recent papers of Papenburg \cite{papenburg2025local}\cite{papenburg2025benjamin}, establishing the global well-posedness for the quasi-periodic BO equation remains an open problem.
\end{remark}

\section{Convergence and Global Well-Posedness via Commuting Flows}

In this section, we establish the global well-posedness of the \eqref{eq:BO}. To evaluate the dynamics and circumvent the spatial derivative loss inherent to standard Picard iteration schemes, we utilize regularized commuting flows. For a spectral parameter $\kappa > 0$, we define the regularized Hamiltonian $H_\kappa$ as:
\begin{equation}
H_\kappa(u) = \frac{1}{2}\left( \kappa P(u) - \kappa^2 \beta(\kappa; u) \right),
\end{equation}
where $P(u) = \frac{1}{2} \langle u, u \rangle_{L^2}$ is the normalized, conserved physical $L^2$-momentum. Because $H_\kappa$ is constructed from the conserved spectral quantities of the Lax hierarchy (as established in \autoref{lemma:H_kappa_properties}), it Poisson-commutes with the BO Hamiltonian. Following the Hamiltonian equations of motion $\partial_t u = -\partial_x \nabla_u H_\kappa$ and evaluating the $L^2$-gradient computed in \autoref{prop:generating_function}, the evolution of $u$ under the regularized flow is governed by:
\begin{equation} \label{eq:V_kappa_def}
\partial_t u = V_\kappa(u) := -\frac{\kappa}{2} \partial_x u + \frac{\kappa^2}{2} \partial_x \Big( m(\kappa, u) + \overline{m}(\kappa, u) - |m(\kappa, u)|^2 \Big).
\end{equation}

We demonstrate the algebraic cancellation in $V_\kappa(u)$, which recovers the BO vector field as $\kappa \to \infty$.

Using the resolvent identity $m = \frac{1}{\kappa} u_+ - \frac{1}{\kappa} R(\kappa, u) \mathcal{L}_u u_+$, we obtain the algebraic decomposition:
\begin{equation}\label{e.q.m_identity}
m(\kappa, u) = \frac{1}{\kappa} u_+ - \frac{1}{\kappa^2} \mathcal{L}_u u_+ + \frac{1}{\kappa^2} R(\kappa, u) \mathcal{L}_u^2 u_+.
\end{equation}
By the defining equation $(\mathcal{L}_u + \kappa)m = u_+$, we have the relation $\mathcal{L}_u m = u_+ - \kappa m = R(\kappa, u) \mathcal{L}_u u_+$. Since the potential $u$ is real-valued ($\overline{u_+} = u_-$), the modulus square $|m|^2$ expands as:
\begin{align}
|m(\kappa, u)|^2 &= \Big(\frac{1}{\kappa} u_+ - \frac{1}{\kappa} \mathcal{L}_u m\Big)\Big(\frac{1}{\kappa} u_- - \frac{1}{\kappa} \overline{\mathcal{L}_u m}\Big) \nonumber \\
&= \frac{1}{\kappa^2} u_+ u_- - \frac{1}{\kappa^2} \Big( u_+ \overline{R(\kappa, u)\mathcal{L}_u u_+} + u_- R(\kappa, u)\mathcal{L}_u u_+ - |R(\kappa, u)\mathcal{L}_u u_+|^2 \Big).
\end{align}
Summing these components and apply \eqref{e.q.m_identity}, the effective gradient evaluates to:
\begin{equation} \label{eq:gradient_expansion}
m + \overline{m} - |m|^2 = \frac{1}{\kappa}u - \frac{1}{\kappa^2}\Big( \mathcal{L}_u u_+ + \overline{\mathcal{L}_u u_+} + u_+ u_- \Big) + r_\kappa(u),
\end{equation}
where the nonlinear remainder $r_\kappa(u)$ is isolated as:
\begin{align} \label{eq:r_kappa_expansion}
r_\kappa(u) = \frac{1}{\kappa^2} \Big( &R(\kappa, u)\mathcal{L}_u^2 u_+ + \overline{R(\kappa, u)\mathcal{L}_u^2 u_+} + u_+ \overline{R(\kappa, u)\mathcal{L}_u u_+} \nonumber \\
&+ u_- R(\kappa, u)\mathcal{L}_u u_+ - |R(\kappa, u)\mathcal{L}_u u_+|^2 \Big).
\end{align}

For the $\mathcal{O}(\kappa^{-2})$ bracket, substituting $\mathcal{L}_u u_+ = -2i\partial_x u_+ + C_+(uu_+)$ and utilizing the Hilbert transform identity $\mathcal{H}u = -i(u_+ - u_-)$ evaluates
\begin{align*}
\mathcal{L}_u u_+ + \overline{\mathcal{L}_u u_+} + u_+ u_-&=-2i\partial_x u_+ + C_+(uu_+)-\overline{2i\partial_x u_+}+\overline{C_+(uu_+)}+ u_+ u_-\\
&=2\mathcal{H}\partial_xu+ C_+(uu_+)+C_-(uu_-)+P_0(u_+u_-)+ u_+ u_-\\
&=2\mathcal{H}\partial_xu+u_+^2+u_-^2+(C_++C_-+I)(u_+u_-)+P_0(u_+u_-)\\
&=2\mathcal{H}\partial_xu+u^2+P_0(u_+u_-).
\end{align*}

Substituting this gradient back into the Hamiltonian vector field $V_\kappa(u)$, the leading-order divergence is eliminated by the term $\frac{\kappa^2}{2} \partial_x (\frac{1}{\kappa} u)$. Furthermore, the constant $P_0(u_+u_-)$ vanishes under the spatial derivative $\partial_x$:
\begin{align} \label{eq:vector_field_cancellation}
V_\kappa(u) &= -\frac{\kappa}{2} \partial_x u + \frac{\kappa^2}{2} \partial_x \left[ \frac{1}{\kappa} u - \frac{1}{\kappa^2} \Big( 2\mathcal{H}\partial_x u + u^2 \Big) + r_\kappa(u) \right] \nonumber \\
&= -\mathcal{H}\partial_{xx} u - \frac{1}{2}\partial_x (u^2) + \mathcal{R}_\kappa(u) := V_{BO}(u) + \mathcal{R}_\kappa(u).
\end{align}
This cancellation confirms that the regularized flow recovers the BO vector field. 

\begin{proposition}
\label{prop:uniform_residual}
Let $\rho>0$, $M>0$, and $\epsilon\in(0,\rho/3)$. There exist $\kappa_0=\kappa_0(M)>0$ and $C_{M,\rho,\epsilon}>0$ such that, for every real-valued $u\in H^{\rho,1}_0$ satisfying $\|u\|_{\rho,1}\le M$, and every $\kappa\ge\kappa_0$, one has
\begin{equation}
\label{eq:uniform_residual}
\|\mathcal R_\kappa(u)\|_{\rho-3\epsilon,1}\le \frac{C_{M,\rho,\epsilon}}{\kappa}.
\end{equation}
\end{proposition}

\begin{proof}
We first record the action of the Lax operator across the analytic scale, \autoref{lemma2.1} and \autoref{lemma2.3} give
\begin{align}
\|\mathcal L_u f\|_{\rho-\epsilon,1}&\le2\|\partial_xf\|_{\rho-\epsilon,1}+\|C_+(uf)\|_{\rho-\epsilon,1} \nonumber\\
&\le C_{M,\rho,\epsilon}\|f\|_{\rho,1}.
\label{eq:Lax_scale_bound}
\end{align}
Since $\|u_+\|_{\rho,1}\le M$, iterating
\eqref{eq:Lax_scale_bound} yields
\begin{equation}
\label{eq:Lax_iterated_bounds}
\|\mathcal L_uu_+\|_{\rho-\epsilon,1}\le C_{M,\rho,\epsilon}M,\qquad\|\mathcal L_u^2u_+\|_{\rho-2\epsilon,1}\le C_{M,\rho,\epsilon}^2M.
\end{equation}

Introduce
\[
A_\kappa(u):=R(\kappa,u)\mathcal L_u^2u_+,\qquad B_\kappa(u):=R(\kappa,u)\mathcal L_uu_+.
\]
\autoref{proposition3.1} and \eqref{eq:Lax_iterated_bounds} imply
$$\|A_\kappa(u)\|_{\rho-2\epsilon,1}\le\frac{2C_{M,\rho,\epsilon}^2M}{\kappa},\qquad\|B_\kappa(u)\|_{\rho-2\epsilon,1}\le\frac{2C_{M,\rho,\epsilon}M}{\kappa}.$$
By \eqref{eq:r_kappa_expansion},
$\kappa^2r_\kappa(u)=A_\kappa+\overline{A_\kappa}+u_+\overline{B_\kappa}+u_-B_\kappa-|B_\kappa|^2$. Thus the Banach algebra estimate gives
\begin{align}\label{eq:r_kappa_full_bound}
\|\kappa^2r_\kappa(u)\|_{\rho-2\epsilon,1}\le \frac{C_{M,\rho,\epsilon}}{\kappa}.
\end{align}
Finally, using $\mathcal R_\kappa(u)=\frac12\partial_x\bigl(\kappa^2r_\kappa(u)\bigr)$, we obtain
\begin{align*}
\|\mathcal R_\kappa(u)\|_{\rho-3\epsilon,1}&\le C_{\epsilon}\|\kappa^2r_\kappa(u)\|_{\rho-2\epsilon,1}\\
&\le\frac{C_{M,\rho,\epsilon}}{\kappa}.
\end{align*}
This proves \eqref{eq:uniform_residual}.
\end{proof}

To apply Picard iteration in the analytic Sobolev space, we establish that the linear propagator forms a strongly continuous group.

\begin{lemma}
\label{lemma:strong_continuity_linear}
For any $\kappa > 0$, the spatial translation operator $S(t) = e^{-\frac{\kappa}{2} t \partial_x}$ is an isometry on $H^{\rho, 1}_0$. Furthermore, for any $f \in H^{\rho, 1}_0$, the map $t \mapsto S(t)f$ is strongly continuous in the analytic topology:
\begin{equation}
\lim_{h \to 0} ||S(t+h)f - S(t)f||_{\rho, 1} = 0.
\end{equation}
\end{lemma}

\begin{proof}
The linear propagator acts on the Fourier coefficients by $\widehat{S(t)f}(n) = e^{-i \frac{\kappa}{2} t n} \hat{f}(n)$. Since the complex exponential has modulus $1$, it commutes with the Sobolev weight, preserving the $H^{\rho, 1}_0$ norm.
To prove strong continuity, we evaluate the difference:
$$||S(t+h)f - S(t)f||_{\rho, 1}^2 = \sum_{n \ne 0} \Big| e^{-i \frac{\kappa}{2} h n} - 1 \Big|^2 \langle n \rangle^2 |\hat{f}(n)|^2 e^{2\rho|n|}.$$
For any fixed $n$, the phase difference tends to $0$ as $h \to 0$. The sequence is bounded by $4 \langle n \rangle^2 |\hat{f}(n)|^2 e^{2\rho|n|}$, which is summable for $f \in H^{\rho, 1}_0$. By the Lebesgue dominated convergence theorem, the limit passes inside the summation, establishing strong continuity in $H^{\rho, 1}_0$.
\end{proof}

\begin{theorem}
\label{thm:H_kappa_flow}
Let $\kappa$ be a sufficiently large, fixed spectral parameter. For any real-valued initial data $u_0 \in H^{\rho, 1}_0$ satisfying the smallness condition \eqref{eq:small_energy_barrier}, the regularized $H_\kappa$ flow admits a unique global solution $u^\kappa \in C(\mathbb{R}; H^{\rho, 1}_0)$. The domain of analyticity is preserved globally, and the spectral energy $\mathcal{E}_\rho(u^\kappa)$ is conserved in $H^{\rho, 1}_0$.
\end{theorem}

\begin{proof}
The integral formulation of the equation of motion generated by $H_\kappa$ is:
\begin{equation}\label{e.q.H_kappa}
u^\kappa(t) = e^{-\frac{\kappa}{2} t \partial_x} u_0 + \frac{\kappa^2}{2} \int_0^t e^{-\frac{\kappa}{2}(t-s)\partial_x} N(u^\kappa(s)) ds,
\end{equation}
where the nonlinear vector field is $N(u) = \partial_x \big(m(\kappa, u) + \overline{m}(\kappa, u) - |m(\kappa, u)|^2\big)$.

To apply Picard iteration in $H^{\rho, 1}_0\cap\mathbb{R}$, we establish that the vector field $N(u)$ is locally Lipschitz. By \autoref{corollary3.1}, the spatial derivative of the gauge variable is expressed through bounded multilinear operations in $H^{\rho, 1}_0\cap\mathbb{R}$:
\begin{equation} \label{eq:dx_m_algebra}
\partial_x m(u) = -\frac{i}{2} \Big( \kappa m(u) + C_+ \big( u m(u) - u \big) \Big).
\end{equation}

For potentials in a fixed ball
\begin{equation*}
B_R:=\{u\in H^{\rho,1}_0\cap\mathbb{R}:\ \|u\|_{\rho,1}\le R\},
\end{equation*}
choose $\kappa$ so large that $\kappa>100C_1R$. The resolvent estimate gives
\begin{equation} \label{eq:m_uniform_bound}
\sup_{u\in B_R}\|m(u)\|_{\rho,1}
\le \frac{R}{\kappa-C_1R}=:M_R.
\end{equation}
By the resolvent identity,
\begin{equation*}
R(\kappa,u)-R(\kappa,v)=-R(\kappa,u)\mathcal{T}_{u-v}R(\kappa,v),
\end{equation*}
and hence
\begin{equation} \label{eq:m_lipschitz}
\|m(u)-m(v)\|_{\rho,1}\le\|R(\kappa,u)(u_+-v_+)\|_{\rho,1}+\|(R(\kappa,u)-R(\kappa,v))v_+\|_{\rho,1}\le L_m\|u-v\|_{\rho,1},
\end{equation}
where $L_m:=\frac{1}{\kappa-C_1R}
+\frac{C_1R}{(\kappa-C_1R)^2}$. The algebraic identity \eqref{eq:dx_m_algebra} then yields
\begin{align*}
\|\partial_xm(u)\|_{\rho,1}&\le \frac12\bigl(\kappa M_R+C_1R(M_R+1)\bigr)=:D_R,\\\|\partial_xm(u)-\partial_xm(v)\|_{\rho,1}
&\le \frac12\bigl(\kappa L_m+C_1(RL_m+M_R+1)\bigr)\|u-v\|_{\rho,1}\\
&=:L_{\partial m}\|u-v\|_{\rho,1}.
\end{align*}
Since
\begin{equation*}
\partial_x|m|^2=(\partial_xm)\overline m+m\,\overline{\partial_xm},
\end{equation*}
the Banach algebra estimate gives
\begin{equation*}
\|\partial_x|m(u)|^2-\partial_x|m(v)|^2\|_{\rho,1}
\le 2C_1\bigl(D_RL_m+M_RL_{\partial m}\bigr)
\|u-v\|_{\rho,1}.
\end{equation*}
Together with the corresponding estimates for $\partial_xm$ and its conjugate, this proves that $N$ is locally Lipschitz on $B_R$.

Since $S(t)=e^{-\frac{\kappa}{2}t\partial_x}$ is a strongly continuous group of isometries by \autoref{lemma:strong_continuity_linear} and $N$ is locally Lipschitz on $H^{\rho,1}_0\cap\mathbb{R}$, the standard contraction argument for semilinear evolution equations applied to \eqref{e.q.H_kappa} yields a unique local solution $u^\kappa \in C([-T^*, T^*]; H^{\rho, 1}_0)$.

Under the smallness condition \eqref{eq:small_energy_barrier}, $\|u^\kappa_+(0)\|_{\rho, 1} \le x_{max}$. Since $u^\kappa \in C((-T^*,T^*);H^{\rho,1}_0)$, the evaluation $t \mapsto \|u^\kappa_+(t)\|_{\rho, 1}$ is continuous. By \autoref{cor:regularized_trapping}, we obtain the global bound:
\begin{equation} \label{eq:native_uniform_bound}
\sup_{|t|<T^*}\|u^\kappa_+(t)\|_{\rho,1}\le x_1,\qquad\sup_{|t|<T^*}\|u^\kappa(t)\|_{\rho,1}\le\sqrt{2}x_1.
\end{equation}
This bound in both time directions prevents finite-time blow-up in $H^{\rho, 1}_0$, ensuring that $T^* = \infty$.
\end{proof}

Denote $M=\sqrt{2}x_1$, then for sufficiently large $\kappa$, $\sup_{t\in\mathbb{R}}\|u^\kappa(t)\|_{\rho,1}\le M$. By Cauchy-Schwarz and the exponential Fourier weight:
\begin{equation} \label{eq:derivative_bound}
||\partial_x u^\kappa(t)||_{L^\infty} \le \sum_{n \ne 0} |n| |\hat{u}^\kappa(n, t)| \le \sum_{n \ne 0} \frac{|n|}{\langle n \rangle} e^{-\rho|n|} \Big( \langle n \rangle |\hat{u}^\kappa(n, t)| e^{\rho|n|} \Big) \le C_\rho ||u^\kappa(t)||_{\rho, 1} \le C_{M}.
\end{equation}

To demonstrate that the family of trajectories $\{u^\kappa(t)\}_{\kappa \ge \kappa_0}$ forms a Cauchy family as $\kappa \to \infty$, we evaluate their convergence in the $L^2$ topology. The difference $w = u^\kappa - u^\nu$ for parameters $\kappa, \nu \ge \kappa_0$ is governed by:
\begin{equation} \label{eq:w_evolution}
\partial_t w = -\mathcal{H}\partial_{xx} w - \frac{1}{2}\partial_x \big( (u^\kappa + u^\nu) w \big) + \big( \mathcal{R}_\kappa(u^\kappa) - \mathcal{R}_\nu(u^\nu) \big).
\end{equation}

Taking the real part of the $L^2$ inner product of \eqref{eq:w_evolution} with $w$, the linear dispersive operator evaluates to zero due to skew-adjointness. Integration by parts transfers the derivative to the bounded potentials \eqref{eq:derivative_bound}:
$$-\frac{1}{2} \operatorname{Re} \left\langle w, \partial_x \big( (u^\kappa + u^\nu) w \big) \right\rangle_{L^2} = -\frac{1}{4} \frac{1}{2\pi} \int_{\mathbb{T}} w^2 \partial_x (u^\kappa + u^\nu) \, dx \le \frac{1}{4} \left\|\partial_x (u^\kappa + u^\nu)\right\|_{L^\infty} ||w||_{L^2}^2 \le \frac{C_M}{2} ||w||_{L^2}^2.$$
Applying the Cauchy-Schwarz inequality to the residual term yields:
$$\langle w,\mathcal{R}_\kappa(u^\kappa) - \mathcal{R}_\nu(u^\nu)\rangle_{L^2}\le\|w\|_{L^2}\left(\|\mathcal{R}_\kappa(u^\kappa)\|_{L^2}+\|\mathcal{R}_\nu(u^\nu)\|_{L^2}\right)\le C_{M,\rho,\epsilon} \left( \frac{1}{\kappa} + \frac{1}{\nu} \right) ||w||_{L^2}.$$
Thus we obtain
$$\begin{aligned}
\frac{1}{2} \frac{d}{dt} ||w||_{L^2}^2 &\le \frac{C_M}{2} ||w||_{L^2}^2 + C_{M,\rho,\epsilon}  \left( \frac{1}{\kappa} + \frac{1}{\nu} \right) ||w||_{L^2}\\
&\le \frac{C_M}{2} ||w||_{L^2}^2 + C_{M,\rho,\epsilon}\left( \frac{1}{\kappa} + \frac{1}{\nu} \right)^2+ C_{M,\rho,\epsilon}||w||^2_{L^2}\\
&\le C_{M,\rho,\epsilon}\|w\|^2_{L^2}+C_{M,\rho,\epsilon}\left( \frac{1}{\kappa} + \frac{1}{\nu} \right)^2.
\end{aligned}$$
Since $w(0) = 0$, Gr\"onwall's lemma confirms that for any finite time $T > 0$ and uniformly for $t \in [-T, T]$:
\begin{equation} \label{eq:L2_cauchy}
||u^\kappa(t) - u^\nu(t)||_{L^2} \le C_{T, M,\rho, \epsilon} \left( \frac{1}{\kappa} + \frac{1}{\nu} \right) \to 0 \quad \text{as} \quad \kappa, \nu \to \infty.
\end{equation}

We upgrade this $L^2$ convergence to the target space $H_0^{\rho - \epsilon, 1}$ by balancing the static loss $\epsilon \in (0, \rho/3)$ with a frequency truncation threshold $J$.

\begin{theorem}
\label{thm:analytic_convergence}
Under the hypotheses of \autoref{thm:H_kappa_flow}, let $u^\kappa$ denote the regularized trajectories with the same initial datum $u_0$. Then for every $\rho > 0$, $T>0$, and $\epsilon \in (0, \rho/3)$. The trajectories $\{u^\kappa(t)\}_{\kappa \ge \kappa_0}$ form a Cauchy family in the space $C([-T, T]; H_0^{\rho-\epsilon, 1})$.
\end{theorem}

\begin{proof}
Let $w(t) = u^\kappa(t) - u^\nu(t)$. We evaluate the norm in $H_0^{\rho-\epsilon, 1}$ by splitting the frequency lattice:
$$||w||_{\rho-\epsilon, 1}^2 = \sum_{0<|n| \le J} \langle n \rangle^2 |\hat{w}(n,t)|^2 e^{2(\rho-\epsilon)|n|} + \sum_{|n| > J} \langle n \rangle^2 |\hat{w}(n,t)|^2 e^{2(\rho-\epsilon)|n|}.$$
For the high-frequency tail, we use the positive-frequency trapping bound. Factoring out the penalty weight $e^{-2\epsilon |n|}$ yields:
$$\sum_{|n| > J} \langle n \rangle^2 |\hat{w}(n, t)|^2 e^{2(\rho-\epsilon)|n|} \le e^{-2\epsilon J} \sum_{|n| > J} \langle n \rangle^2 |\hat{w}(n, t)|^2 e^{2\rho|n|} \le e^{-2\epsilon J} ||w(t)||_{\rho, 1}^2.$$
By \eqref{eq:native_uniform_bound}, the high-frequency contribution is bounded by $8x_1^2 e^{-2\epsilon J}$, which decays to zero as $J \to \infty$, independent of $\kappa$ and $\nu$. 

For the low-frequency part, we bound the exponential weight:
$$\sum_{0<|n| \le J} \langle n \rangle^2 |\hat{w}(n, t)|^2 e^{2(\rho-\epsilon)|n|} \le \langle J \rangle^2 e^{2(\rho-\epsilon) J} \sum_{0<|n| \le J} |\hat{w}(n, t)|^2 \le \langle J \rangle^2 e^{2(\rho-\epsilon) J} ||w(t)||_{L^2}^2.$$

To establish the Cauchy property in $C([-T, T]; H_0^{\rho-\epsilon, 1})$, let $\eta > 0$ be given. We select an integer $J = J(\eta, \epsilon, x_1)$ sufficiently large such that $8x_1^2 e^{-2\epsilon J}<\eta^2/2$. With $J$ fixed, \eqref{eq:L2_cauchy} shows that the family $\{u^\kappa\}_{\kappa \ge \kappa_0}$ is Cauchy in $L^2_0(\mathbb{T})$ on $[-T, T]$. There exists $K = K(J, \eta) > 0$ such that for all $\kappa, \nu >K$ and $t \in [-T, T]$, $\langle J \rangle^2 e^{2(\rho-\epsilon) J}||w(t)||_{L^2}^2 < \eta^2 /2$. 

Summing the components yields $||u^\kappa(t) - u^\nu(t)||_{\rho-\epsilon, 1}^2 < \eta^2$. Taking the supremum over $t \in [-T, T]$ establishes that $\{u^\kappa\}_{\kappa \ge \kappa_0}$ is a Cauchy family in $C([-T, T]; H_0^{\rho-\epsilon, 1})$.
\end{proof}

\begin{proposition}
\label{prop:weak_limit_BO_solution}
Assume the regularized initial data $u_0\in H^{\rho,1}_0\cap\mathbb{R}$ satisfy the smallness condition \eqref{eq:small_energy_barrier}, and let $x_1$ be the stable root given by \autoref{cor:regularized_trapping}. For every $T>0$ and $\epsilon\in(0,\rho/3)$, $u^\kappa$ converge, as $\kappa\to\infty$, to a unique solution $u$ satisfying
\begin{equation} \label{eq:weak_limit_solution_class}
u\in C([-T,T];H^{\rho-\epsilon,1}_0)\cap C^1([-T,T];H^{\rho-3\epsilon,1}_0).
\end{equation}
Moreover,
\begin{equation} \label{eq:limit_endpoint_bound}
\|u_+(t)\|_{\rho,1}\le x_1,\qquad\|u(t)\|_{\rho,1}\le\sqrt{2}x_1
\end{equation}
for every $t\in\mathbb{R}$.
\end{proposition}

\begin{proof}
By \autoref{thm:analytic_convergence}, for each $T>0$ and $\epsilon\in(0,\rho/3)$ there exists $u\in C([-T,T];H^{\rho-\epsilon,1}_0)$ such that
\begin{equation} \label{eq:regularized_lower_radius_limit}
u^\kappa\to u
\quad\text{in }C([-T,T];H^{\rho-\epsilon,1}_0).
\end{equation}

To show the uniqueness, let $u,v$ be the solution of \eqref{eq:weak_limit_solution_class} with the same initial datum, and set $w=u-v$, then
\begin{equation*}
\partial_tw=-\mathcal{H}\partial_{xx}w
-\frac12\partial_x((u+v)w).
\end{equation*}
Taking the real $L^2$ inner product with $w$, using skew-adjointness of $\mathcal{H}\partial_{xx}$ and integration by parts, gives
\begin{equation} \label{eq:weak_uniqueness_energy}
\frac12\frac{d}{dt}\|w\|_{L^2}^2
=-\frac14\frac{1}{2\pi}\int_{\mathbb{T}}
\partial_x(u+v)w^2\,dx.
\end{equation}
For $g\in H^{\rho-\epsilon,1}$,
\begin{equation*}
\|\partial_xg\|_{L^\infty}
\le \left(\sum_{n\ne0}\frac{n^2}{1+n^2}e^{-2(\rho-\epsilon)|n|}\right)^{1/2}\|g\|_{\rho-\epsilon,1}.
\end{equation*}
Thus $\|\partial_x(u+v)\|_{L^\infty}$ is bounded on $[-T,T]$. Gronwall's inequality applied to \eqref{eq:weak_uniqueness_energy} proves uniqueness.

The regularized evolution is
\begin{equation*}
\partial_tu^\kappa=V_{BO}(u^\kappa)+\mathcal{R}_\kappa(u^\kappa),\qquad V_{BO}(v)=-\mathcal{H}\partial_{xx}v-\frac12\partial_x(v^2).
\end{equation*}
The Cauchy estimates and the Banach algebra property show that
\begin{equation*}
V_{BO}:H^{\rho-\epsilon,1}_0\longrightarrow H^{\rho-3\epsilon,1}_0
\end{equation*}
is continuous. Hence
\begin{equation*}
V_{BO}(u^\kappa)\to V_{BO}(u)
\quad\text{in }C([-T,T];H^{\rho-3\epsilon,1}_0).
\end{equation*}
The uniform residual estimate \eqref{eq:uniform_residual}, applied on the ball supplied by \eqref{eq:native_uniform_bound}, gives
\begin{equation*}
\mathcal{R}_\kappa(u^\kappa)\to0
\quad\text{in }C([-T,T];H^{\rho-3\epsilon,1}_0).
\end{equation*}
Let $\kappa\to\infty$ in the integral equation $u^\kappa(t)=u_0+\int_0^t[V_{BO}(u^\kappa(s))+\mathcal{R}_\kappa(u^\kappa(s))]ds$, then
$$u(t)=u_0+\int_0^tV_{BO}(u(s))ds,$$
thus $u$ belongs to the $C^1([-T,T],H^{\rho-3\epsilon,1}_0)$ and solves \eqref{eq:BO}.

Finally, for every fixed $t$ and every $n$, \eqref{eq:regularized_lower_radius_limit} implies
\begin{equation*}
\widehat{u^\kappa}(n,t)\to\widehat u(n,t).
\end{equation*}
Fatou's lemma and \eqref{eq:native_uniform_bound} yield
\begin{align*}
\|u_+(t)\|_{\rho,1}^2&\le\liminf_{\kappa\to\infty}\|u^\kappa_+(t)\|_{\rho,1}^2\le x_1^2.
\end{align*}
\end{proof}

\begin{lemma}
\label{lem:endpoint_weak_continuity}
Under the hypotheses of \autoref{prop:weak_limit_BO_solution},
\begin{equation*}
u^\kappa(t)\rightharpoonup u(t)\quad\text{in }H_0^{\rho,1}
\end{equation*}
for every $t\in\mathbb{R}$, and
\begin{equation}\label{eq:CW}
u\in C_w(\mathbb{R};H_0^{\rho,1}).
\end{equation}
\end{lemma}

\begin{proof}
For fixed $t$, the sequence $u^\kappa(t)$ is bounded in $H_0^{\rho,1}$, this gives a weak convergent subsequence $u^{\kappa_j}\rightharpoonup v$. The continuous embedding $H_0^{\rho,1}\hookrightarrow H^{\rho-\epsilon,1}_0$ shows that the same subsequence converges weakly to $v$ in $H^{\rho-\epsilon,1}_0$. By \eqref{eq:regularized_lower_radius_limit}, it converges strongly to $u(t)$, so $v=u(t)$. Thus as $\kappa\to\infty$,
$$u^\kappa(t)\rightharpoonup u(t)\quad\text{in }H_0^{\rho,1}.$$

To prove weak time continuity, let $t_j\to t$ and $\varphi\in H_0^{\rho,1}$. Let $P_N$ denote the Fourier projection onto $0<|n|\le N$. The endpoint bound gives
\begin{align*}
|\langle u(t_j)-u(t),\varphi\rangle_{\rho,1}|
&\le|\langle P_N(u(t_j)-u(t)),P_N\varphi\rangle_{\rho,1}|+|\langle(I-P_N)(u(t_j)-u(t)),(I-P_N)\varphi\rangle_{\rho,1}|\\
&\le|\langle P_N(u(t_j)-u(t)),P_N\varphi\rangle_{\rho,1}|+2\sqrt{2}x_1\|(I-P_N)\varphi\|_{\rho,1}.
\end{align*}
Choose $N$ so that $2\sqrt{2}x_1\|(I-P_N)\varphi\|_{\rho,1}\le\frac{\delta}{2}$ for small $\delta$. For fixed $N$, 
$$\|P_N(u(t_j)-u(t))\|_{\rho,1}\le e^{2\epsilon N}\|u(t_j)-u(t)\|_{\rho-\epsilon,1}\to0,$$ which makes the first term tend to zero. This proves \eqref{eq:CW}.
\end{proof}

\begin{proposition}
\label{prop:limit_spectral_energy_conservation}
Assume the stronger smallness condition \eqref{eq:main_small_data}. Then the limiting BO solution satisfies
\begin{equation} \label{eq:limit_energy_conservation}
\mathcal{E}_\rho(u(t))=\mathcal{E}_\rho(u_0)
\qquad\text{for every }t\in\mathbb{R}.
\end{equation}
\end{proposition}

\begin{proof}
Set $A_0:=\mathcal{E}_\rho(u_0)^{1/2}<f(r_*)$, and let $X_*<r_*$ be the smaller solution of $f(X_*)=A_0$. \autoref{prop:weak_limit_BO_solution} implies
\begin{equation} \label{eq:strict_ball_trapping}
\|u^\kappa(t)\|_{\rho,1}\le \sqrt{2}X_*,\qquad\|u(t)\|_{\rho,1}\le \sqrt{2}X_*
\end{equation}
for all $t$ and all sufficiently large $\kappa$. Fix $t\in\mathbb{R}$. By \autoref{lem:endpoint_weak_continuity},
\begin{equation*}
u^\kappa(t)\rightharpoonup u(t)\quad\text{in }H_0^{\rho,1}.
\end{equation*}
All states in \eqref{eq:strict_ball_trapping} lie in $\overline B_{\sqrt{2}r_*}(H_0^{\rho,1})$. Weak lower semicontinuity from \autoref{thm:spectral_strict_convexity}, together with conservation of the spectral energy along the regularized flow, gives
\begin{equation} \label{eq:energy_one_sided}
\mathcal{E}_\rho(u(t))\le\liminf_{\kappa\to\infty}\mathcal{E}_\rho(u^\kappa(t))=\mathcal{E}_\rho(u_0).
\end{equation}

Fix $t_0\in\mathbb{R}$ and restart the regularized construction from the datum $u(t_0)$. By \eqref{eq:energy_one_sided},
\begin{equation*}
\mathcal{E}_\rho(u(t_0))^{1/2}\le A_0<f(r_*),
\end{equation*}
and \eqref{eq:strict_ball_trapping} gives $\|u_+(t_0)\|_{\rho,1}<X_*$. Applying \eqref{eq:energy_one_sided} yields
\begin{equation*}
\mathcal{E}_\rho(u(t_0+s))
\le\mathcal{E}_\rho(u(t_0)).
\end{equation*}
Taking $s=-t_0$ gives
\begin{equation}\label{eq:energy_another_sided}
\mathcal{E}_\rho(u_0)
\le\mathcal{E}_\rho(u(t_0)).
\end{equation}
On the other hand, $s\mapsto u(t_0+s)$ is also a BO solution with initial datum $u(t_0)$, and this belongs to the uniqueness class of \autoref{prop:weak_limit_BO_solution}. Combining \eqref{eq:energy_another_sided} with \eqref{eq:energy_one_sided} proves \eqref{eq:limit_energy_conservation}.
\end{proof}

\begin{corollary}
\label{cor:endpoint_strong_continuity}
Under \eqref{eq:main_small_data}, \eqref{eq:BO} satisfies
\begin{equation*}
u\in C(\mathbb{R};H^{\rho,1}_0).
\end{equation*}
\end{corollary}

\begin{proof}
Let $t_j\to t$. By \autoref{lem:endpoint_weak_continuity},
\begin{equation*}
u(t_j)\rightharpoonup u(t)
\quad\text{in }H_0^{\rho,1}.
\end{equation*}
By \autoref{prop:limit_spectral_energy_conservation},
\begin{equation*}
\mathcal{E}_\rho(u(t_j))=\mathcal{E}_\rho(u(t)).
\end{equation*}
All states lie in $\overline B_{\sqrt{2}r_*}(H_0^{\rho,1})$ by \eqref{eq:strict_ball_trapping}. The Kadec property in \autoref{thm:spectral_strict_convexity} therefore gives
\begin{equation*}
u(t_j)\to u(t)
\quad\text{strongly in }H_0^{\rho,1}.
\end{equation*}
\end{proof}

\begin{theorem}
\label{thm:endpoint_continuous_dependence}
For every $T>0$, the data-to-solution map is continuous from the small-data set $\Omega_*$ defined in \autoref{thm:main_theorem} into $C([-T,T];H^{\rho,1}_0)$.
\end{theorem}

\begin{proof}
Let $u_{0,n}\to u_0$ in $H^{\rho,1}_0$, with $u_{0,n},u_0\in\Omega_*$, and let $u_n,u$ be the corresponding BO solutions. By \autoref{prop:spectral_coordinates_near_identity}, the spectral energy is continuous on $H_0^{\rho,1}\cap\mathbb{R}$, so
\begin{equation*}
A_n:=\mathcal{E}_\rho(u_{0,n})^{1/2}\to
A_0:=\mathcal{E}_\rho(u_0)^{1/2}<f(r_*).
\end{equation*}
Choose $\bar r$ such that the smaller root of $f(x)=A_0$ is less than $\bar r<r_*$. Since $f$ is strictly increasing on $[0,x_{max}]$, for all sufficiently large $n$,
\begin{equation} \label{eq:uniform_data_family_trapping}
\sup_{t\in\mathbb{R}}\bigl(\|(u_n)_+(t)\|_{\rho,1}\bigr)\le\bar r\qquad\sup_{t\in\mathbb{R}}\|u_+(t)\|_{\rho,1}\le \bar r.
\end{equation}

Set $w_n=u_n-u$. The $L^2$ energy estimate used in \autoref{prop:weak_limit_BO_solution}, together with \eqref{eq:uniform_data_family_trapping}, yields
\begin{equation} \label{eq:limit_L2_stability}
\sup_{|t|\le T}\|w_n(t)\|_{L^2}
\le C_{T,\rho,\bar r}\|u_{0,n}-u_0\|_{L^2}
\longrightarrow0.
\end{equation}
For any $\epsilon\in(0,\rho)$ and any integer $J\ge1$,
\begin{align*}
\|w_n(t)\|_{\rho-\epsilon,1}^2
\le{}\langle J\rangle^2e^{2(\rho-\epsilon)J}\|w_n(t)\|_{L^2}^2+e^{-2\epsilon J}\|w_n(t)\|_{\rho,1}^2.
\end{align*}
The last norm is uniformly bounded by \eqref{eq:uniform_data_family_trapping}. Choosing first $J$ and then $n$, using \eqref{eq:limit_L2_stability}, gives
\begin{equation} \label{eq:family_lower_radius_convergence}
u_n\to u\quad\text{in }C([-T,T];H^{\rho-\epsilon,1}_0).
\end{equation}

We next remove the $\epsilon$. Assume, by contradiction, that convergence fails in $C([-T,T];H_0^{\rho,1})$. Then there are $\eta>0$, a subsequence, and $t_n\in[-T,T]$ such that
\begin{equation*}
\|u_n(t_n)-u(t_n)\|_{\rho,1}\ge\eta.
\end{equation*}
After passing to a subsequence, $t_n\to t_*$. From \eqref{eq:family_lower_radius_convergence} and continuity of $u$ in the smaller space,
\begin{equation*}
u_n(t_n)\to u(t_*)\quad\text{in }H^{\rho-\epsilon,1}_0.
\end{equation*}
The uniform endpoint bound then implies
\begin{equation*}
u_n(t_n)\rightharpoonup u(t_*)\quad\text{in }H_0^{\rho,1}.
\end{equation*}
Energy conservation and continuity of the initial spectral energy give
\begin{equation*}
\mathcal{E}_\rho(u_n(t_n))=\mathcal{E}_\rho(u_{0,n})\to\mathcal{E}_\rho(u_0)=\mathcal{E}_\rho(u(t_*)).
\end{equation*}
By \autoref{cor:endpoint_strong_continuity}, 
$$u(t_n)\to u(t_*)\quad \text{in } H_0^{\rho,1}.$$
Hence by triangle inequality,
\begin{equation*}
\|u_n(t_n)-u(t_n)\|_{\rho,1}\to0,
\end{equation*}
a contradiction. Thus $u_n\to u$ in $C([-T,T];H_0^{\rho,1})$, which proves the asserted convergence in $H^{\rho,1}_0$.
\end{proof}

\begin{proof}[Proof of \autoref{thm:main_theorem}]
The smallness condition \eqref{eq:main_small_data} implies the regularized trapping condition \eqref{eq:small_energy_barrier}. Therefore, \autoref{prop:weak_limit_BO_solution} constructs a unique global BO solution in every smaller analytic topology. The stable root satisfies $X(u_0)<r_*$, so \autoref{prop:limit_spectral_energy_conservation} and \autoref{cor:endpoint_strong_continuity} give
\begin{equation*}
u\in C(\mathbb{R};H^{\rho,1}_0)\cap C^1(\mathbb{R};H^{\rho-3\epsilon,1}_0)
\end{equation*}
for every $\epsilon\in(0,\rho/3)$, together with spectral energy conservation and the claimed uniform bounds. Finally, \autoref{thm:endpoint_continuous_dependence} proves continuity of the data-to-solution map in the endpoint topology.
\end{proof}

\section*{Acknowledgments}
The author would like to thank Professor Qingtang Su for comments, suggestions, and corrections to an earlier version of this manuscript. The author was partially supported by NSFC 12494541 and NSFC 12421001. The observation of convexity of $\mathcal{E}_\rho(u)$ was made by GPT5.6sol, which help to remove the derivative loss of $\epsilon$, which strengthen the v1 conclusion (\autoref{prop:weak_limit_BO_solution}) of this manuscript.

\bibliographystyle{abbrv}
\bibliography{reference}

\end{document}